\documentclass[a4paper,11pt]{amsart}

\usepackage{path, graphicx}
\usepackage{amsmath}
\usepackage{amssymb,amsfonts,amsthm}           % rajout de symboles + mise en forme theoremes
\usepackage{enumerate}                          % enumerations
\usepackage{array,delarray}                     % matrices et tableaux
\usepackage{xspace}

\newtheorem{prop}{Proposition}

\newtheorem{defn}[prop]{Definition}
\newtheorem{notation}[prop]{Notation}
\newtheorem{lem}[prop]{Lemma}
\newtheorem{rmk}[prop]{Remark}

\newcommand{\A}[1]{{\vphantom{A}}_{#1}{A}}

\newcommand{\Id}{\mbox{Id}}

\newcommand{\vecteur}[1]{\mathnormal{\boldsymbol{#1}}}
\newcommand\vectq{\vecteur{q}}                                            % Vecteurs

\newcommand\vectf{\vecteur{f}} 
 
\newcommand{\vectW}[1][]{\vecteur{W}_{\!\!#1\,}}

\newcommand{\vectY}[1][]{\vecteur{Y}_{\!\!#1\,}}

\newcommand{\trans}[1]{{#1}^{\texttt{t}}}
\newcommand{\contr}{\;{\boldsymbol \odot}\;}

\newcommand{\qy}{{q_y}}							% moments conserves a l'equilibre
\newcommand{\qx}{{q_x}}
\newcommand{\qz}{{q_z}}

\newcommand{\fluxx}{{\varphi_x}}					 % moments
\newcommand{\fluxy}{{\varphi_y}}
\newcommand{\tenseurxx}{{\varphi_{xx}}}
\newcommand{\tenseurxy}{{\varphi_{xy}}}

\newcommand{\e}{\varepsilon}
\newcommand{\ecarre}{{\varepsilon_2}}
\newcommand{\ecube}{\varepsilon_3}
\newcommand{\xxe}{xx\varepsilon}
\newcommand{\xecarre}{x\varepsilon_2}
\newcommand{\yecarre}{y\varepsilon_2}

\newcommand{\mfluxx}{m_{\fluxx}}					 % m_moments
\newcommand{\mfluxy}{m_\fluxy}
\newcommand{\mtenseurxx}{m_\tenseurxx}
\newcommand{\mtenseurxy}{m_\tenseurxy}
\newcommand{\me}{m_\e}
\newcommand{\mecarre}{m_{\ecarre}}
\newcommand{\mecube}{m_{\ecube}}
\newcommand{\mxecarre}{m_{\xecarre}}
\newcommand{\myecarre}{m_{\yecarre}}
\newcommand{\mxxe}{m_{\xxe}}

\newcommand{\sfluxx}{s_\fluxx}					 % temps relaxation
\newcommand{\sfluxy}{s_\fluxy}
\newcommand{\stenseurxx}{s_{\tenseurxx}}
\newcommand{\stenseurxy}{s_{\tenseurxy}}
\newcommand{\se}{s_{\e}}
\newcommand{\sxecarre}{s_{\xecarre}}
\newcommand{\syecarre}{s_{\yecarre}}
\newcommand{\secarre}{s_{\ecarre}}
\newcommand{\secube}{s_{\ecube}}
\newcommand{\sxxe}{s_{\xxe}}

\newcommand{\sigmafluxx}{\sigma_\fluxx}					 % sigma
\newcommand{\sigmafluxy}{\sigma_\fluxy}
\newcommand{\sigmatenseurxx}{\sigma_{\tenseurxx}}
\newcommand{\sigmatenseurxy}{\sigma_{\tenseurxy}}
\newcommand{\sigmae}{\sigma_{\e}}
\newcommand{\sigmaxecarre}{\sigma_{\xecarre}}
\newcommand{\sigmayecarre}{\sigma_{\yecarre}}
\newcommand{\sigmaecarre}{\sigma_{\ecarre}}

\newcommand{\eerho}{E_\e^\rho}					 % matrice equilibre
\newcommand{\eeqx}{E_\e^\qx}
\newcommand{\eeqy}{E_\e^\qy}
\newcommand{\eeqz}{E_\e^\qz}
\newcommand{\eecarrerho}{E_{\ecarre}^\rho}
\newcommand{\eecarreqx}{E_{\ecarre}^\qx}
\newcommand{\eecarreqy}{E_{\ecarre}^\qy}
\newcommand{\efluxxrho}{E_{\fluxx}^\rho}
\newcommand{\efluxxqx}{E_{\fluxx}^\qx}
\newcommand{\efluxxqy}{E_{\fluxx}^\qy}
\newcommand{\efluxyrho}{E_\fluxy^\rho}
\newcommand{\efluxyqx}{E_\fluxy^\qx}
\newcommand{\efluxyqy}{E_\fluxy^\qy}
\newcommand{\etenseurxxrho}{E_\tenseurxx^\rho}
\newcommand{\etenseurxxqx}{E_\tenseurxx^\qx}
\newcommand{\etenseurxxqy}{E_\tenseurxx^\qy}
\newcommand{\etenseurxyrho}{E_\tenseurxy^\rho}
\newcommand{\etenseurxyqx}{E_\tenseurxy^\qx}
\newcommand{\etenseurxyqy}{E_\tenseurxy^\qy}

\newcommand{\exxerho}{E_{\xxe}^\rho}
\newcommand{\exxeqx}{E_{\xxe}^\qx}
\newcommand{\exxeqy}{E_{\xxe}^\qy}
\newcommand{\execarrerho}{E_{x\ecarre}^\rho}
\newcommand{\execarreqx}{E_{x\ecarre}^\qx}
\newcommand{\execarreqy}{E_{x\ecarre}^\qy}
\newcommand{\eyecarrerho}{E_{y\ecarre}^\rho}
\newcommand{\eyecarreqx}{E_{y\ecarre}^\qx}
\newcommand{\eyecarreqy}{E_{y\ecarre}^\qy}
\newcommand{\eecuberho}{E_{\ecube}^\rho}
\newcommand{\eecubeqx}{E_{\ecube}^\qx}
\newcommand{\eecubeqy}{E_{\ecube}^\qy}

\newcommand{\dis}{\displaystyle}
\newcommand{\guill}[1]{\textquotedblleft{#1}\textquotedblright}
\newcommand{\pt}{\mbox{\begin{LARGE}$\cdot$\end{LARGE}}}
\newcommand{\qq}{\forall\,}
\newcommand{\rg}{S_{n,d,N}}

\newcommand{\DQ}[2]{\textup{$\textsf{D}_{#1}\textsf{Q}_{#2}$}\xspace}	   % DdQq
\newcommand{\ddqq}{\DQ{d}{q}}
\newcommand{\ddqn}{\DQ{2}{9}}
\newcommand{\ddqt}{\DQ{2}{13}}
\newcommand{\dtqv}{\DQ{3}{27}}
\newcommand{\dtqd}{\DQ{3}{19}}

\newcommand{\eqq}[1]{\text{\tt #1}}
\newcommand{\grandO}{{\mathcal O}}
\newcommand{\rotation}[1][d]{SO_{#1}({\mathbb R})}

\usepackage{color}

\newcommand{\ppviden}[1]{$\text{I}_{\text{#1}}$}
\newcommand{\ppvidet}[1]{$\text{I}_{\text{#1}}^{\prime}$}
\newcommand{\ppn}[1]{{\upshape(\ppviden{#1})}}
\newcommand{\ppns}[2]{{\upshape(\ppviden{#1}--\ppviden{#2})}}
\newcommand{\ppt}[1]{{\upshape(\ppvidet{#1})}}
\newcommand{\ppts}[2]{{\upshape(\ppvidet{#1}--\ppvidet{#2})}}\newlength{\taillemargehypo}
\newlength{\taillelabelhypo}
\newenvironment{prpn}[1]%
{%
  \itshape%
  \begin{list}%
    {{\upshape(\ppviden{#1})\hfill\null}}%
    {\setlength{\labelwidth}{\taillelabelhypo}%
     \setlength{\leftmargin}{\taillemargehypo}%
     \setlength{\itemsep}{0pt}%
     \setlength{\parsep}{0pt}%
     \setlength{\topsep}{0.25\baselineskip}%
    }%
  \item
}%
{%
  \end{list}%
}%
\newenvironment{prpt}[1]%
{%
  \itshape%
  \begin{list}%
    {{\upshape(\ppvidet{#1})\hfill\null}}%
    {\setlength{\labelwidth}{\taillelabelhypo}%
     \setlength{\leftmargin}{\taillemargehypo}%
     \setlength{\itemsep}{0pt}%
     \setlength{\parsep}{0pt}%
     \setlength{\topsep}{0.25\baselineskip}%
    }%
  \item
}%
{%
  \end{list}%
}%

\usepackage{setspace}
\graphicspath{{ADG_figures/}}
\makeatletter
\def\subsection{\@startsection{subsection}{2}%
	\z@{.5\linespacing\@plus.7\linespacing}{.25\linespacing}%
	{\normalfont\bfseries}}
\def\subsubsection{\@startsection{subsubsection}{3}%
	\z@{.5\linespacing\@plus.7\linespacing}{.25\linespacing}%
	{\normalfont\itshape}}
\makeatother

\title[Linear LB schemes: parameters choices and isotropy properties]{Linear Lattice Boltzmann Schemes for Acoustic: \\ 
parameters choices and isotropy properties}
\author[orsay]{Adeline Augier}\email{Adeline.Augier@math.u-psud.fr}
\author[orsay,cnam]{Fran\c{c}ois Dubois}\email{Francois.Dubois@math.u-psud.fr}
\author[orsay]{Benjamin Graille}\email{Benjamin.Graille@math.u-psud.fr}
\address[orsay]{Universit\'e Paris-Sud, Laboratoire de Math\'ematiques, UMR 8628, Orsay, F-91405, France}
\address[cnam]{Conservatoire National des Arts et M\'etiers, Department of mathematics, Paris, France}

\begin{document}

\begin{abstract}

In this paper, we investigate the numerous parameters choi\-ces for linear lattice Boltzmann schemes according to the definition of the isotropic order given in \cite{ADG11}. 
This property---written in a general framework including all of the \DQ{d}{q} schemes---can be read through a group operation. 
It implies some relations on the parameters of the scheme (equilibrium states and relaxation times) that give rigorous methodology to select them according to the desired order of isotropy. 
For acoustic applications in two spaces dimensions (namely \ddqn and \ddqt schemes) this methodology is used to propose a full description of the sets of parameters that involve isotropy of order $m$ ($m\in\{1,2,3,5\}$ for \ddqn and $m\in\{1,2\}$ for \ddqt).
We then propose numerical illustrations for the \ddqn scheme.
\end{abstract}
\maketitle

%%%%%%%%%%%%%%%%%%%%%%%
\section*{Introduction}
%%%%%%%%%%%%%%%%%%%%%%%
\label{sec:intro}
Lattice Boltzmann schemes \cite{H94,LL00,D05} are numerically very interesting because of their efficiency and then represent a promising field in computation fluid dynamics \cite{CWSD92,YGL05}. However, since the directions of the lattice are privileged, these schemes are \textit{a priori} not isotropic even though they are used to solve isotropic phenomena. All the more, \cite{DL11} mentions that isotropy also improves the lattice Boltzmann scheme in the sense of the stability.  Thus, the isotropy properties of the \ddqq schemes (where $d$ represents the space dimension and $q$ the number of discrete velocities) are the aim of numerous previous works though isotropy is not defined in the same way: for instance, isotropy of the stress tensor required by fluid equations in \cite{CWSD92}, isotropy by solving an eigenvalues problem in \cite{LL00}.

As it is explained in \cite{ADG11}, because of the criss-cross pattern of the scheme, it is natural to look for parameters that give the same behavior on each axes. By extend, isotropy can be related to the invariance of the spatial frame by all the special orthogonal transformations. Space transformations are actually considered in \cite{vdSE99,YGL05,CGO08} for discrete cases: isometries of the Bravais lattice in \cite{ vdSE99}, discrete rotations in \cite{YGL05,CGO08}. An abstract way is considered in \cite{CGO08} in order to obtain a systematic procedure to construct higher-order BGK models. 

In \cite{ADG11}, we give a rigorous definition for a \DQ{d}{q} scheme to be isotropic. This definition is based on an invariance of the equivalent equations with regard to all the special orthogonal transformations of the space and can be used for every scheme. Moreover, it gives a procedure to precise the parameters (namely equilibrium states and relaxation times) of linear and orthogonalized schemes in order to improve them in the sense of the isotropy. This methodology finds some set of coefficients again that are already known thanks to the experiment. 

In the first Section, we recall some notations about the linear \ddqq scheme defined with orthogonalized moments and then the definition of the equivalent equations that come from  a Taylor expansion of the lattice Boltzmann scheme \cite{D05}. 

In the second Section, we briefly recall the definition of the isotropy given in \cite{ADG11}. A \DQ{d}{q} scheme is said isotropic at order $m$ if the equivalent equations of order $m$ are isotropic. This definition is written thanks to a group operation and gives an easy systematic procedure to precise some of the parameters of the linear lattice Boltzmann scheme. The end of this Section consists in the calculation of the degrees of freedom that we have to take into account for every scheme. 

  In the third Section, we focus on acoustic applications using the \ddqn and \ddqt schemes. We first recall the definition of this particular schemes (moments and equilibrium states) and the methodology of the formal computation. Then we give a total description of the sets of parameters that involve isotropy of order $m$ ($m$ from 1 to 5 for the \ddqn scheme and to 2 for the \ddqt scheme). Concerning the \ddqt scheme, we also propose a selected choice of sets of parameters that involve isotropy at third order.

In the last Section, preliminary numerical results are discussed to illustrate the different orders of isotropy for the \ddqn scheme.

%%%%%%%%%%%%%%%%%%%%%%%%%%%%%%%%%%%%%%%%%%%%%%%%%%%%%%%%%%
\section{Lattice Boltzmann method and equivalent equations}
%%%%%%%%%%%%%%%%%%%%%%%%%%%%%%%%%%%%%%%%%%%%%%%%%%%%%%%%%%
In this Section, we first recall some notations about the lattice Boltzmann method and more precisely we focus on a linear LB scheme with orthogonalized moments. We then give the formal equivalent equations of order $m$ that is the partial differential equations with which the scheme is consistent at order $m$, for all integer $m$.

%%%%%%%%%%%%%%%%%%%%%%%%%%%%%%%%%%%%%%%%%%%%%%%%%%%%%%%%%%
\subsection{Notations for the lattice Boltzmann method}
\label{sec:notations}
%%%%%%%%%%%%%%%%%%%%%%%%%%%%%%%%%%%%%%%%%%%%%%%%%%%%%%%%%%

We use the notation proposed by d'Humi\`eres in \cite{H94}: we consider a regular lattice in $d$ dimensions $\mathcal{L}$ with typical mesh size $\Delta x$. The time step $\Delta t$ is determined thanks to the velocity scale $\lambda$ by the relation:  
\begin{equation*}
\Delta t=\dfrac{\Delta x}{\lambda}.
\end{equation*}
For the $q$-velocities scheme denoted by \ddqq, we introduce $V=(v_j)_{0\leq j\leq q-1}$ the set of $q$ velocities and we assume that for each node $x$ of $\mathcal{L}$, and  each $v_j$ in $V$, the point $x+v_j\Delta t$ is also a node of the lattice $\mathcal{L}$. 

The aim of the \ddqq scheme is to precise the particle distribution $\vectf=\trans{(f_j)}_{0\leq j\leq q-1}$ for $x\in\mathcal{L}$ and discrete values of time $t$ by solving the following PDEs:
\begin{equation}
\label{edp}
\dfrac{\partial f_j}{\partial t}+v_j\cdot\nabla f_j=-\dfrac{1}{\tau_j}(f_j-f_j^{\mbox{eq}}),\quad 0\leq j\leq q-1,
\end{equation}
where $f_j^{\mbox{eq}}$ describes the distribution $f_j$ at the equilibrium and $\tau_j$ is the relaxation time (related on $f_j$). 

The interest of the \ddqq scheme lives in the method of resolution of \eqref{edp}: we proceed in two steps like in a splitting method as explained below.

\begin{itemize}

\item First, we are interested in the relaxation $\vectf\rightarrow\vectf^*$ step that consists in solving  
\begin{equation}
\label{collision}
\dfrac{\partial f_j^*}{\partial t}=-\dfrac{1}{\tau_j}(f_j-f_j^{\mbox{eq}}),\quad 0\leq j\leq q-1.
\end{equation}
In order to solve \eqref{collision}, we first consider the moments at equilibrium. These moments are divided into two types: the ones that are conserved at equilibrium denoted by $\vectW\in\mathbb{R}^{N}$ and the ones that are not conserved at equilibrium denoted by $\vectY\in\mathbb{R}^{q-N}$. More precisely, because of the acoustic applications, the moments $\vectW$ and $\vectY$ are linear combinations of the distribution functions $\vectf$ (i.e. there exists an invertible matrix $M=(M_{ij})_{0\leq i,j\leq q-1}$ such that $\trans{(\trans{\vectW},\trans{\vectY})}=M\,\vectf$). The size and the definition of the matrix $M$ depend on the scheme, however for every \ddqq scheme the $N$ first lines of $M$ are the same. In fact we have $\vectW=\trans{(\rho,\trans{\vectq})}$ where the total density $\rho$ and the momentum $\vectq$ are given by 
\begin{equation*}
\vectW[0]=\rho=\sum_{j=0}^{q-1}f_j
\qquad \mbox{and}\qquad 
\vectW[\alpha] = \vectq_\alpha= \sum_{j=0}^{q-1}v_j^\alpha f_j,\ 1\leq \alpha\leq d.
\end{equation*}
Now, we define the moments after the relaxation thanks to the relations:
\begin{equation}
\label{moment_eq}
\left\{\begin{array}{rcll}
\vectW[k]^*&=&\vectW[k],&0\leq k\leq N-1\\
\vectY[k]^*&=&\vectY[k]+s_k(\vectY[k]^{\mbox{eq}}-\vectY[k]),&N\leq k\leq q-1,
\end{array}\right.
\end{equation}
where, for $N\leq k\leq q{-}1$, $s_k$ is related to all the $\tau_j$, $N\leq j\leq q{-}1$ and $\vectY[k]^{\mbox{eq}}$ is the moment $\vectY[k]$ at equilibrium. For stability reasons, we choose $s_k$ such that $0<s_k<2$, $N\leq k\leq q{-}1$. In order to simplify the calculations, we introduce the reals $\sigma_k$, $N\leq k\leq q{-}1$, by
\begin{equation}
\label{eq:sigma}
 \sigma_k\ =\ \dfrac{1}{s_k}-\dfrac{1}{2},\quad \sigma_k>0.
\end{equation}
As the scheme is supposed to be linear, the moments at equilibrium $\vectY^{\mbox{eq}}$ linearly depend on the conserved moments $\vectW$ through a matrix denoted by $E$, so that equation~\eqref{moment_eq} reads:
\begin{equation*}
\left(\begin{array}{c}
         \vectW^*\\
	 \vectY^*
        \end{array}
\right) = J \left(\begin{array}{c}
         \vectW\\
	 \vectY
        \end{array}\right),
\qquad \text{with} \quad
J =  \left(\begin{array}{cc}
   \Id_{N}&0\\
    S\,E&\Id_{q-N}-S
  \end{array}\right),
\end{equation*}
where $S$ is the diagonal matrix of the relaxation times $s_k$, $N\leq k\leq q{-}1$. Finally, we reconstruct the distribution after the relaxation with $f^*\ =\ M^{-1}\,\trans{(\trans{\vectW^*},\trans{\vectY^*})} $. 

\item Second, we solve the transport step  
\begin{equation}
\label{advection}
\dfrac{\partial f_j}{\partial t}+v_j\cdot\nabla f_j=0,\quad 0\leq j\leq q-1,
\end{equation}
thanks to the characteristics method: 
\begin{equation}
f_j(x+v_j\Delta t,t+\Delta t)\ = \ f_j^*(x,t),\qquad \forall\,x\in\mathcal{L},\ 0\leq j\leq q-1.
\end{equation}

\end{itemize}

Finally a time step of the linear lattice Boltzmann scheme \ddqq reduces to:
\begin{equation}
 \label{ddqq}
 \vectf(x+v_j\Delta t,t+\Delta t)\  =\ M^{-1}\,J\,M\,\vectf(x,t),\qquad \qq x\in\mathcal{L},\ 0\leq j\leq q-1. 
\end{equation}

%%%%%%%%%%%%%%%%%%%%%%%%%%%%%%%%%%%%%%%%%%%%%%%%%%
\subsection{Equivalent equations}
%%%%%%%%%%%%%%%%%%%%%%%%%%%%%%%%%%%%%%%%%%%%%%%%%%
\label{sec:equivalent_equations}

By convention, we denote $\pt^i$ the $i$th line and $\pt_j$ the $j$th column of a tensor $\pt$. 
By convention, a Latin letter is an index related to the moments (of size $N$) while Greek letter represents an index related to the space dimension (of size $d$). We then introduce the system of $N$ equivalent equations \cite{D05} of order $m$ with which the linear lattice Boltzmann scheme is consistent at order $m$ (that is to say that the rest is $\grandO(\Delta t^m)$):

\begin{equation}
 \label{equivalent_equations}
\partial_t\vectW^i+ \sum_{n=1}^m\A{n}^i  \contr \nabla^n\vectW\ =\ 0,\ 1\leq i\leq N,
\end{equation}

where $\A{n}^i\in\rg,\ 1\leq n\leq m,\ 1\leq i\leq N$ are tensors of order $n+2$ that take into account the coefficients of $E$ and $S$. 
The space $\rg$ is defined as the quotient space $\rg=\mathbb{R}^{N^2d^n}/\sim$, where two tensors are equivalent for the equivalence relation $\sim$ if the associate partial differential operators are the same (using Schwarz's theorem). By convention, the maximal contraction operator $\contr$ is defined by, $ 1\leq i\leq N,\ 1\leq n\leq m$: 
 \begin{equation*}
 \left(\A{n}\contr\nabla^n\vectW\right)^i\ :=\ \A{n}^i\contr\nabla^n\vectW\ :=\ \dis\sum_{\substack{1\leq j\leq N\\ 1\leq \alpha_1,\cdots,\alpha_n\leq d}}\A{n}_j^{i,\alpha_1,\cdots,\alpha_n}\partial_{\alpha_1}\cdots\partial_{\alpha_n}\vectW^j. 
 \end{equation*}
These equivalent equations come from a formal calculus explained in \cite{DL11} and an algorithm that is easy to use is described in \cite{Dpreparation}.

%%%%%%%%%%%%%%%%%%%%%%%%%%%%%%%%%%%%%%%%%%%%%%%%%%
\section{Isotropy condition}
%%%%%%%%%%%%%%%%%%%%%%%%%%%%%%%%%%%%%%%%%%%%%%%%%%
In this Section, we first briefly recall the construction of the definition of isotropy for a \ddqq scheme given in \cite{ADG11}. The \ddqq scheme is characterized through the set of its equivalent equations. And as these equations are PDEs, it is natural to derive this definition from the definition of an isotropic system of PDEs.
Consequently the isotropy reads as a group operation and involves a set of equations on the coefficients of the matrices $E$ and $S$. The second step of this Section consists in counting the number of degrees of freedom given by the isotropy. 

%%%%%%%%%%%%%%%%%%%%%%%%%%%%%%%%%%%%%%%%%%%%%%%%%%
\subsection{Isotropy and Algebra}
\label{sec:isotropy}

We first recall the definition of an isotropic system of PDEs:
\begin{defn}
\label{def:isotropyPDE}
 The system of PDEs
\begin{equation}
 \label{PDE}
\partial_t\vectW^i+ \sum_{n=1}^M\A{n}^i\contr\nabla^n\vectW\ =\ 0,\ 1\leq i\leq N,
\end{equation}
is said isotropic if it is invariant by special orthogonal transformation of the frame. 
\end{defn}

For $r$ an special orthogonal transformation of the frame in $d$ spaces dimensions, we define the orthogonal matrix $R(r)$ by
\begin{equation*}
R(r) := \left(\begin{array}{cc}1&0\\ 0&r\end{array}\right),
\end{equation*}
such that $R(r)^{-1}\vectW$ is the vector of the conserved moments in the new frame. 

\begin{rmk}
In this study, we focus on the group of special orthogonal transformations in two dimensions that is the group of rotations $\rotation[2]$. It can be parametrized with one real number corresponding to the rotation angle, so that $r$ reads
\begin{equation*}
\left(\begin{array}{cc}\cos(\theta)&-\sin(\theta)\\ \sin(\theta)&\cos(\theta)\end{array}\right),\quad \theta\in\mathbb{R}.
\end{equation*}
\end{rmk}

As it is proven in \cite{ADG11}, the isotropy property can be read through a group operation defined in the following Definition. Furthermore, this definition gives a set of equations that are really easy to use and that give some relations between the parameters of the scheme. 
\begin{defn}
Let $n$ be in $\mathbf{N}^*$. Then the group operation $\Phi_n\ :\ \rotation\times\rg\rightarrow\rg$ is defined for $1\leq i,j\leq N$, $1\leq\alpha_1,\cdots,\alpha_n\leq d$, $1\leq n\leq m$, by the relation:
\begin{multline}
\label{def:Phi_n}
\left(\Phi_n(r)(\A{n})\right)^{i,\alpha_1,\cdots,\alpha_n}_j:=\\ \sum_{{1\leq \beta_1,\cdots,\beta_n\leq d}}\,\sum_{{1\leq k,l\leq N}}\left(R(r)\right)^i_l\A{n}^{l,\beta_1,\cdots,\beta_n}_kr^{\alpha_1}_{\beta_1}\cdots r^{\alpha_n}_{\beta_n} (R(r)^{-1})^k_j.
\end{multline}
\end{defn}
We then obtain a characterization for a PDE to be isotropic (for the proof see \cite{ADG11}):
\begin{prop}
\label{prop:isotropy}
 Let \eqref{PDE} be a PDE of order $m$. It is isotropic if the tensor $\A{n}$ is a fixed point of $\Phi_n$, $1\leq n\leq m$, that is $\Phi_n(r)(\A{n})=\A{n}$, $\qq r\in\rotation$, $1\leq n\leq m$. 
\end{prop}

Finally, we recall the definition for a lattice Boltzmann scheme to be isotropic \cite{ADG11}:
\begin{defn}
\label{def:isotropy}
A lattice Boltzmann scheme is said isotropic at order $m$ if the system of equivalent equations (\ref{equivalent_equations}) at order $m$ is isotropic. \\
Furthermore, we denote by $L_N(r)\ :=\ \sum_{1\leq n\leq N}(\Phi_n(r)(\A{n})-\A{n}){\triangle t}^n$ the lack of isotropy at order $N$ for the special orthogonal transformation $r$.
\end{defn}

%%%%%%%%%%%%%%%%%%%%%%%%%%%%%%%%%%%%%%%%%%%%%%%%%%%%%%%%%%%%%%%%%%%%%%%%
\subsection{Consequences on the degrees of freedom}
%%%%%%%%%%%%%%%%%%%%%%%%%%%%%%%%%%%%%%%%%%%%%%%%%%%%%%%%%%%%%%%%%%%%%%%%
\label{sec:ddl}
In this Subsection, we count the degrees of freedom that have to be taken into account for studying  isotropy. The total number of the parameters that can be freely chosen is indeed reduced by considering the relations imposed by the physics.

Let us make the count of the equations and  of the unknowns:

\begin{itemize}
 \item There are $N^2\,(d+n-1)!/(n!(d-1)!)$ different components of $\A{n}$ each one corresponding to a relation for isotropy thanks to Proposition~\ref{prop:isotropy} (\cite{ADG11}). 
 \item Then, we take into account the equilibrium states, namely the coefficients of the matrix $E$ defined in Section~\ref{sec:notations}. The matrix $E$ is usually chosen sparse due to the experiments and to the physical \guill{good sense}. It is remarkable, that the isotropy condition for the \ddqq scheme could give some justifications of certain traditional choices of coefficients. However, in this paper, $E$ is chosen full in order to give as degrees of freedom as possible, given then $N(q-N)$ additional parameters. 
 \item It remains to consider the relaxations times.  Actually, as it is recalled for example in \cite{ADG11}, there exists essentially three types of lattice Boltzmann schemes according to the values of the relaxation times: the Bhatnagar - Gross - Krook (BGK) scheme, the Two Relaxation Times (TRT) scheme and the Multiple Relaxation Times (MRT) scheme. Since our purpose is precise the relaxation times in order to improve the behavior of the scheme in the sense of the isotropy, we consider the third case that is the most general (even if it is not the most often used). Thus, we consider $q-N$ independent unknown parameters on $S$. 
\end{itemize}

Finally, as we consider the most general \ddqq scheme in the linear case, there are $(q-N)(N+1)$ free parameters and $N^2(d+n-1)!/(n!(d-1)!)$ equations given by the isotropy conditions. 
However, all of these parameters do not have to be precise by the isotropy conditions: some are linked with some physical properties through equivalent equations \cite{D05}. 

\begin{itemize}
 \item The linear equilibrium state of the energy described by $\eerho\rho\lambda^2+\eeqx\qx\lambda^2+\eeqy\qy\lambda^2+\eeqz\qz\lambda^2$ is such that the density contribution $\eerho$ is linked with the sound velocity $c_0$. For example, as it is recalled in \cite{LL00,LL03,DL09,DL11}, we have $\eerho=6c_0^2-4$ for the \ddqn scheme, $\eerho=26c_0^2-28$ for the \ddqt scheme, $\eerho=57c_0^2/\lambda^2-30$ for the \dtqd scheme, and $\eerho=3c_0^2-2$ for the \dtqv scheme.  
In fact, all of these relations are given by the physicals properties that imply the values $\A{1}^{21}_1=\A{1}^{32}_1=c_0^2$, namely the contribution of $\partial \rho/\partial x$ (respectively $\partial \rho/\partial y$) in the conservation equation that describes the evolution of the momentum $q_x$ (respectively $q_y$). 
\item Furthermore, since there exists a \ddqn scheme consistent with the Navier-Stokes equations (see \cite{D05}), a link shall be done between both the viscosities of the considered fluid and two of the relaxation times (see for example \cite{DL11,LL00}). Namely, if the shear viscosity is denoted by $\zeta$ and the bulk viscosity by $\mu$, in the case of the \ddqn scheme, we get
\begin{equation}
\label{eq:viscosities}
\mu=\tfrac{1}{3}\lambda\,\Delta x\,\sigmatenseurxx\mbox{ and }\zeta=\lambda\,\Delta x\,\sigmae \bigl(\tfrac{5}{9}-c_0^2\bigr),
\end{equation}
where $c_0^2$ is the sound velocity, $\sigmatenseurxx$ (respectively $\sigmae$) is related to the relaxation time depending on the moment $\mtenseurxx$ (respectively $\me$) thanks to \eqref{eq:sigma}, where each moment is defined below. 

 \item For acoustic applications, our objective is to conserve the freedom of both viscosities even if they are sometimes taken equal. In other words, we are looking for a set of coefficients that does not link $\se$ and $\stenseurxx$.  

 \item Moreover, since the moments $\me$ and $\mtenseurxx$ are of the same order (namely 2), we have to take into account a lattice Boltzmann scheme with multiple relaxation times for acoustic applications (a Two Relaxations Times scheme leads to a link between $\mu$ and $\zeta$ because $\se$ and $\stenseurxx$ are equal). 
\end{itemize}

Finally it remains $(q-N)(N+1)-3$ parameters to be precise.

%%%%%%%%%%%%%%%%%%%%%%%%%%%%%%%%%%%%%%%%%%%%%%%%%%
\section{Isotropy for acoustic applications in two dimensions}
%%%%%%%%%%%%%%%%%%%%%%%%%%%%%%%%%%%%%%%%%%%%%%%%%%

The aim of this section is to investigate the property of isotropy at different orders for two schemes often used for acoustic applications (\ddqn and \ddqt) and in particular to determine all of the possible choices of coefficients that yield an isotropic scheme of order~$m$.
We first recall the definition of the \ddqn and the \ddqt schemes through their moments. 
After specifying the degrees of freedom, we explain the used method to determine all of the solutions of these huge non linear systems.
In the second part of this section, we explicate the results order after order for the \ddqn scheme and in the third part, we precise all of the obtained results for the \ddqt scheme, the intricacy of the equations increasing so much with the order that we have to restrict ourselves to the second order.

%%%%%%%%%%%%%%%%%%%%%%%%%%%%%%%%%%%%%%%%%%%%%%%%%%
\subsection{Generalities and methodology}
%%%%%%%%%%%%%%%%%%%%%%%%%%%%%%%%%%%%%%%%%%%%%%%%%%
In this subsection, we specify the notations for the \ddqn and the \ddqt schemes and we propose a methodology to solve the very large systems of non linear equations that appear when the isotropy is investigate at each order. These systems are written in terms of the coefficients of the matrices $E$ (describing the equilibrium states) and $S$ (describing the relaxation times) and have to be satisfy for every rotation angle $\theta$.

For both considered lattice Boltzmann schemes (\ddqn and \ddqt), three moments are conserved during the collision step: the density $\rho$ and the two coordinates of the macroscopic momentum $q_x$ and $q_y$. They are defined by
\begin{equation*}
\rho=\sum_{j=0}^{q-1}f_j,
\quad
q_x=\sum_{j=0}^{q-1}v_j^xf_j,
\quad
q_y=\sum_{j=0}^{q-1}v_j^yf_j.
\end{equation*}
Next, we introduce six moments that are not conserved during the collision step: 
the kinetic energy $\me$, the square of the kinetic energy $\mecarre$, the coordinates of the heat flux $\mfluxx$ and $\mfluxy$, and two moments of order two $\mtenseurxx$ and $\mtenseurxy$. They read
\begin{gather*}
\me=\sum_{j=0}^{q-1}\tfrac{1}{2}|v_j|^2f_j,
\quad
\mecarre=\sum_{j=0}^{q-1}\left(\tfrac{1}{2}|v_j|^2\right)^2f_j,
\\
\mfluxx=\sum_{j=0}^{q-1}\tfrac{1}{2}|v_j|^2v_j^xf_j,
\quad
\mfluxy=\sum_{j=0}^{q-1}\tfrac{1}{2}|v_j|^2v_j^yf_j,
\\
\mtenseurxx=\sum_{j=0}^{q-1} \left(\left(v_j^x\right)^2- \left(v_j^y\right)^2\right)f_j,
\quad
\mtenseurxy=\sum_{j=0}^{q-1} v_j^xv_j^yf_j.
\end{gather*}

In this paper, we consider the linear \ddqn scheme obtained after orthogonalization of these nine moments, with a full matrix $E$ describing the equilibrium states and a full diagonal matrix $S$ describing the relaxation times that are written
\begin{equation}
\label{matrices_generales_d2q9}
 E :=
\begin{pmatrix}
\eerho\lambda^2         &  \eeqx\lambda          &  \eeqy\lambda \\
\eecarrerho\lambda^4    &  \eecarreqx\lambda^3   &  \eecarreqy\lambda^3 \\
\efluxxrho\lambda^3     &  \efluxxqx\lambda^2    &  \efluxxqy\lambda^2 \\
\efluxyrho\lambda^3     &  \efluxyqx\lambda^2    &  \efluxyqy\lambda^2 \\
\etenseurxxrho\lambda^2 &  \etenseurxxqx\lambda  &  \etenseurxxqy\lambda \\
\etenseurxyrho\lambda^2 &  \etenseurxyqx\lambda  &  \etenseurxyqy\lambda 
\end{pmatrix}, \quad
 %\mbox{and}\ 
 S:=\operatorname{Diag}(\se,\secarre,\sfluxx,\sfluxy,\stenseurxx,\stenseurxy).
\end{equation}

In order to define the \ddqt scheme, we have to introduce four additional moments:
the cube of the kinetic energy $\mecube$, two moments of order five $\mxecarre$ and $\myecarre$, and finally a moment of order four $\mxxe$.
\begin{gather*}
 \mecube=\sum_{0\leq j\leq q-1}\left(\tfrac{1}{2}|v_j|^2\right)^3f_j,
 \quad
 \mxecarre=\sum_{0\leq j\leq q-1} v_j^x\left(\tfrac{1}{2}|v_j|^2\right)^2 f_j,
 \\
 \myecarre=\sum_{0\leq j\leq q-1} v_j^y\left(\tfrac{1}{2}|v_j|^2\right)^2 f_j,
 \\
 \mxxe=\sum_{0\leq j\leq q-1}\left(\tfrac{1}{2}|v_j|^2\right)\left((v_j^x)^2-(v_j^y)^2\right) f_j.
\end{gather*}

In this paper, we consider the linear \ddqt scheme obtained after orthogonalization of these thirteen moments, with a full matrix $E$ describing the equilibrium states and a full diagonal matrix $S$ describing the relaxation times that are written 
 \begin{gather}
 \notag
 E :=
\begin{pmatrix}
\eerho\lambda^2\  &  \eeqx\lambda  &  \eeqy\lambda\\
\etenseurxxrho\lambda^2 &  \etenseurxxqx\lambda  &  \etenseurxxqy\lambda \\
\etenseurxyrho\lambda^2 &  \etenseurxyqx\lambda  &  \etenseurxyqy\lambda \\
\efluxxrho\lambda^3        &  \efluxxqx\lambda^2        &  \efluxxqy\lambda^2\\
\efluxyrho\lambda^3        &  \efluxyqx\lambda^2        &  \efluxyqy\lambda^2\\
\execarrerho\lambda^5        &  \execarreqx\lambda^4        &  \execarreqy\lambda^4\\
\eyecarrerho\lambda^5        &  \eyecarreqx\lambda^4        &  \eyecarreqy\lambda^4\\
\eecarrerho\lambda^4 &   \eecarreqx\lambda^3               &  \eecarreqy\lambda^3\\
\eecuberho\lambda^6 &   \eecubeqx\lambda^5                & \eecubeqy\lambda^5 \\
\exxerho\lambda^4   &   \exxeqx\lambda^3		 &   \exxeqy\lambda^3
\end{pmatrix},\\
 \label{sd2q13}
S := \operatorname{Diag}(\se,\stenseurxx,\stenseurxy,\sfluxx,\sfluxy,\sxecarre,\syecarre,\secarre,\secube,\sxxe).
\end{gather}

We then apply Definition~\ref{def:isotropy} on the \ddqn scheme in order to describe all of the sets of parameters that improve these schemes (in the sense of the isotropy). In two space dimensions, the special orthogonal transformations $r$ are the well-known rotation matrices, $N=3$ and $m\leq5$. 
We first establish the degrees of freedom for each schemes at each order.
Section~\ref{sec:ddl} gives the number of equations (this number does not depend on the scheme): 
$9(n+1)$ where $n$ is the considered order. The explicit calculation of these numbers is given in Table~\ref{table:nb_generals_equations}. 
Furthermore, there are $21$ unknowns parameters for the \ddqn scheme and $37$ for the \ddqt scheme.

\begin{table}[htbp]
\begin{center}
 \begin{tabular}{|c||c|c|c|c|}
\hline
order		&1	&2	&3	&4\\
\hline 
polynomial equations	&18	&27	&36	&45\\
\hline
\end{tabular} 
\caption{Number of polynomials eqs. for two dimensions}
\label{table:nb_generals_equations}
\end{center}
\end{table}

We remark that all of these equations are polynomial in $\cos \theta$ and $\sin\theta$, where $\theta$ is the angle that parametrizes the rotation. The study of their coefficients 
then gives other equations that only depend on the coefficients of $S$ and $E$. For instance, for the \ddqn scheme, at first order we get $18$ polynomial equations thanks to Proposition~\ref{prop:isotropy} and the formula:
$$\A{1}^{i\alpha}_j = \sum_{{1\leq \beta\leq d}}\,\sum_{{1\leq k,l\leq N}}\left(R(r)\right)^i_l\A{1}^{l,\beta}_kr^{\alpha}_{\beta}(R(r)^{-1})^k_j,\ 1\leq i,j\leq 3,\ 1\leq\alpha\leq2. $$
All of these equations shall be filed as follow:
\begin{itemize}
 \item 6 of them are null.
 \item 8 of them are of type $a\cos^3\theta+b\cos^2\theta\sin\theta+c\cos\theta+d\sin\theta+e=0$. Since the functions $\theta\mapsto\cos^3\theta$, $\theta\mapsto\cos^2\theta\sin\theta$, $\theta\mapsto\cos\theta$, $\theta\mapsto\sin\theta$, and $\theta\mapsto1$ are free,
each of these equations implies 5 additional equations: $a=b=c=d=e=0$. 
 \item 4 of them are of type $a\cos\theta\sin\theta+b\sin^2\theta=0$. Since the functions $\theta\mapsto\cos\theta\sin\theta$ and $\theta\mapsto\sin^2\theta$ are free, 
 each of these equations implies 2 additional equations: $a=b=0$. 
\end{itemize}

The isotropy for the \ddqn scheme at first order is characterized by 48 equations:
8 coefficients of the matrix $E$ are then fixed and it remains $21{-}8=13$ parameters for the isotropy at second order.

Following this methodology order after order,
we are able to list in Table~\ref{table:d2q9} and \ref{table:d2q13} for each order the number of equations that we have to take into account (namely, equations that do not depend on the angle $\theta$) and the number of parameters that we have to precise. 

\begin{table}[htbp]
\begin{center}
\begin{minipage}{0.45\textwidth}
\begin{center}
\begin{tabular}{|c||c|c|c|c|}
\hline
order		&1	&2	&3	&4\\
\hline 
equations	&48	&78	&116	&169\\
\hline
parameters	&21	&13	&8	&6\\
\hline
\end{tabular} 
\caption{Number of eqs. for \ddqn \label{table:d2q9}}
\end{center}
\end{minipage}
\hfill
\begin{minipage}{0.45\textwidth}
\begin{center}
\begin{tabular}{|c||c|c|c|}
\hline
order		&1	&2	&3\\
\hline 
equations	&48	&78	&148\\
\hline
parameters	&37	&29	&23\\
\hline
\end{tabular} 
\caption{Number of eqs. for \ddqt \label{table:d2q13}}
\end{center}
\end{minipage}
\end{center}
\end{table}

\begin{rmk}
Both numbers of equations for the \ddqn and \ddqt schemes are the same 
(the order $n$ being fixed)
because the tensors $\A{n}$ have the same size. 
However, the coefficients of these polynomial equations (in $\sin(\theta)$ and $\cos(\theta)$)  depend on the geometry of the scheme and as a result the number of equations we have to solve is not the same for the two schemes. 
\end{rmk}

The methodology to investigate the isotropy at order $n$ is then resume in three steps.
\begin{enumerate}
 \item compute the equivalent equations of order $n$ (they could be obtained using a formal code \cite{DL11,Dpreparation}) of the lattice Boltzmann scheme described by the matrices $E$ and $S$ taking into account the relations obtained at previous orders.
 \item write the whole of the equations given by Proposition~\ref{prop:isotropy}, the number of equations being given in Table~\ref{table:d2q9} and Table~\ref{table:d2q13}.
 \item precise the relations on the coefficients that have to be satisfied in order to have the desired order of isotropy.
\end{enumerate}

Finally, we introduce a notation that yields to identify the equations: 
\begin{notation}[\eqq{eqa1...anj}]
A natural way to identify the equations consists in specify the number of the conservation equation and the corresponding coefficient of the tensor $\A{N}$, $1\leq N\leq m$. Then, we denote by \eqq{eqia1...anj} the equation given thanks to $\A{n}^{i,a_1,\cdots,a_n}_j$. 
We then denote the coefficient of $\cos^k(\theta)\sin^l(\theta)$ in the equation \eqq{eqa1...anj} by \eqq{eqia...anjcosksinl}.
\end{notation}

\subsection{Results for the \ddqn scheme}
\label{sec:d2q9}

In this Subsection, the main result on the \ddqn scheme is given in Proposition~\ref{prop:d2q9}. The proof is then subdivided into five Lemmas by proceeding order after order.

\begin{prop}
\label{prop:d2q9}
 Let $L_5(r)$ be the lack of isotropy for the \ddqn scheme at fifth order for the rotation $r$, then we get the following propositions.
\begin{itemize}
 \item The scheme is isotropic at first order, that is $L_5(r)=\grandO({\Delta t}^2),\qq r\in\rotation[2]
$, iff $\eeqx$, $\eeqy$, $\etenseurxxrho$, $\etenseurxxqx$, $\etenseurxxqy$, $\etenseurxyrho$, $\etenseurxyqx$, and $\etenseurxyqy$ are zero.
 \item The scheme is isotropic at second order, that is $L_5(r)=\grandO({\Delta t}^3),\qq r\in\rotation[2]$,
iff it is isotropic at first order,
$\efluxxrho$, $\efluxxqy$, $\efluxyrho$, $\efluxyqx$ are zero, and $\efluxxqx=\efluxyqy=(\sigmatenseurxx-4\sigmatenseurxy)/(2\sigmatenseurxy+\sigmatenseurxx)$.
 \item The scheme is isotropic at third order, that is $L_5(r)=\grandO({\Delta t}^4),\qq r\in\rotation[2]$,
iff the three following properties are satisfied:
\begin{itemize}
\item it is isotropic at second order,
\item $\sigmatenseurxx=\sigmatenseurxy$ and $\efluxxqx=-1$,
\item at least one of these three properties is satisfied
\begin{itemize}
\item $2\eecarrerho+4+3\eerho=0$, $\eecarreqx=\eecarreqy=0$,
\item $\sigmafluxx=\sigmafluxy=1/(12\sigmatenseurxx)$, $\eecarreqx=\eecarreqy=0$,
\item $\sigmafluxx=\sigmafluxy=1/(12\sigmatenseurxx)$,  $\sigmatenseurxx=\sigmae$.
\end{itemize}
\end{itemize}
 \item The scheme is isotropic at fourth order, that is $L_5(r)=\grandO({\Delta t}^5),\qq r\in\rotation[2]$,
iff the three following properties are satisfied:
\begin{itemize}
\item it is isotropic at third order,
\item $2\eecarrerho+4+3\eerho=0$, $\sigmae=\sigmatenseurxx$, $\eecarreqx=\eecarreqy=0$, $\sigmafluxx=\sigmafluxy=1/(6\sigmatenseurxx)$,
\item $2+3\eerho=0$ or $\sigmaecarre=\sigmatenseurxx$. 
\end{itemize}
 \item The scheme is never isotropic at fifth order, that is to say the property $L_5(r)=\grandO({\Delta t}^6),\qq r\in\rotation[2]$, is never true. 
\end{itemize}
\end{prop}

\begin{rmk}
Some of these properties that insure the isotropy of the \ddqn scheme are well-known and often used.  
In particular, the parameters that are taken null in order to obtain the isotropy at first and second orders are also null for isotropic reasons considering the kinetic solution at equilibrium in the continuous environment. However, the results on the third and fourth orders are more surprising even if some usual \ddqn schemes can be seen as peculiar cases of those we propose.
\end{rmk}

The proof of Proposition~\ref{prop:d2q9} is detailed below thanks to five Lemmas: one for each considered order. The followed methodology is given in the previous subsection. 
We have to remark that we continually use the fact that the coefficients $\sigmae$, $\sigmaecarre$, $\sigmafluxx$, $\sigmafluxy$, $\sigmatenseurxx$, and $\sigmatenseurxy$ are positive.

%%%%%%%%%%%%%%%%%%%%%%%%%%%%%%%%%%
\subsubsection{First order}
%%%%%%%%%%%%%%%%%%%%%%%%%%%%%%%%%%

The isotropy at first order is described in the following Lemma.
\begin{lem}
 \label{lem:d2q9o1}
The \ddqn scheme is isotropic at first order if, and only if the properties \ppns{1}{2} are satisfied.
\begin{prpn}{1}
At equilibrium, the energy is proportional to the density (the proportionality factor is relative to the sound velocity).
\end{prpn}
\begin{prpn}{2}
Both moments $\mtenseurxx$ and $\mtenseurxy$ vanish at equilibrium.
\end{prpn}
Properties \ppns{1}{2} involve the following structure for the matrix $E$ (there is no constraint on the matrix $S$ at first order):
\begin{equation*}
E =
\begin{pmatrix}
\eerho\lambda^2         &  0                     &  0 \\
\eecarrerho\lambda^4    &  \eecarreqx\lambda^3   &  \eecarreqy\lambda^3 \\
\efluxxrho\lambda^3     &  \efluxxqx\lambda^2    &  \efluxxqy\lambda^2 \\
\efluxyrho\lambda^3     &  \efluxyqx\lambda^2    &  \efluxyqy\lambda^2 \\
0                       &  0                     & 0 \\
0                       &  0                     &  0
\end{pmatrix}.
\end{equation*}
\end{lem}

\begin{rmk}
This result gives a rigorously justification for usual physical assumptions. 
\end{rmk}

\begin{proof}
Considering eqs. \eqq{eq211cos1sin1} and  \eqq{eq211cos0sin2} yields to 
$\etenseurxxrho=\etenseurxyrho=0$. Then the system of four equations 
\begin{equation*}
 \left\{\begin{array}{rcl}
	\eqq{eq222cos2sin1}&=&0,\\
	\eqq{eq222cos1sin0} &=&0,\\
	\eqq{eq222cos0sin1}&=&0,\\
	\eqq{eq222cos0sin0}&=&0,
        \end{array}\right.
\end{equation*}
gives that the four unknowns $\etenseurxxqx$, $\etenseurxxqy$, $\etenseurxyqx$ and $\etenseurxyqy$ vanish.

To solve all of the remained equations, it is sufficient to consider the equations \eqq{eq212cos0sin0} and \eqq{eq322cos0sin1} in the unknowns $\eeqx$ and $\eeqy$: more precisely the equilibrium energy is characterized by $\eeqx=\eeqy=0$.  
\end{proof}

%%%%%%%%%%%%%%%%%%%%%%%%%%%%%%%%%%
\subsubsection{Second order}
%%%%%%%%%%%%%%%%%%%%%%%%%%%%%%%%%%

For the \ddqn scheme, the isotropy at second order can be characterized by the following Lemma:

\begin{lem}
\label{lem:d2q9o2}
The \ddqn scheme is isotropic at second order if and only if the properties \ppns{1}{4} are satisfied.
\begin{prpn}{3}
 At equilibrium, the heat flux is proportional to the momentum, that is $\mfluxx=\efluxxqx\lambda^2\qx$ and $\mfluxy=\efluxxqx\lambda^2\qy$.
\end{prpn}
\begin{prpn}{4}
 The proportionality factor $\efluxxqx$ is related to the relaxation times by the relation $\efluxxqx=-(4\sigma_\tenseurxy-\sigma_\tenseurxx)/(\sigmatenseurxx+2\sigmatenseurxy).$
\end{prpn}
Properties \ppns{1}{4} involve the following structure for the matrix $E$ (there is no constraint on the matrix $S$ at second order):
\begin{equation*}
E = \begin{pmatrix}
\eerho\lambda^2      &  0                    &  0 \\
\eecarrerho\lambda^4 &  \eecarreqx\lambda^3  &  \eecarreqy\lambda^3 \\
0                    &  c\lambda^2           &  0 \\
0                    &  0                    &  c\lambda^2 \\
0                    &  0                    & 0 \\
0                    &  0                    &  0
\end{pmatrix},
\end{equation*}
with $c=-(4\sigmatenseurxy-\sigmatenseurxx)/(\sigmatenseurxx+2\sigmatenseurxy)$. 
\end{lem}

\begin{proof}
Considering first eq. \eqq{eq2221} and more precisely the system
\begin{equation*}
 \left\{\begin{array}{rcl}
  \eqq{eq2221cos0sin0}&=&0,\\
  \eqq{eq3111cos0sin0}&=&0,
 \end{array}\right.
\end{equation*}
we obtain that $\efluxxrho$, $\efluxyrho$ vanish. Then the system
\begin{equation*}
 \left\{\begin{array}{rcl}
         \eqq{eq2222cos0sin2}&=&0,\\
	\eqq{eq3112cos0sin2}&=&0,
        \end{array}\right.
\end{equation*}
implies both results $\efluxxqx=\efluxyqy$ and  $\efluxyqx=-\efluxxqy$. Using eq.~\eqq{eq3113cos0sin2} yields to $\efluxxqy=0$ and eq.~\eqq{eq2233cos2sin2} gives the characterization on $\efluxxqx$. Finally, we prove that the previous relations are sufficient to impose the isotropy at second order.
\end{proof}

\begin{rmk}
The relation $\efluxxqx=-(4\sigmatenseurxy-\sigmatenseurxx)/(\sigmatenseurxx+2\sigmatenseurxy)$ is exactly the relation (41) of \cite{LL00} except that the roles of the relaxation times $\sigmatenseurxx$ and $\sigmatenseurxy$ have to be exchanged. 
\end{rmk}
%%%%%%%%%%%%%%%%%%%%%%%%%%%%%%%%%%
\subsubsection{Third order}
%%%%%%%%%%%%%%%%%%%%%%%%%%%%%%%%%%

For the \ddqn scheme, the isotropy at third order can be characterized by the following Lemma:

\begin{lem}
\label{lem:d2q9o3}
The \ddqn scheme is isotropic at third order if and only if the whole of the properties \ppns{1}{5} and one of the properties \ppns{6}{8} are satisfied.
\begin{prpn}{5}
 The relaxation times $\sigma_\tenseurxx$ and $\sigma_\tenseurxy$ (relating to second order moments) are the same, so that $c=-1$. 
\end{prpn}
\begin{prpn}{6}
At equilibrium, the square of the energy is proportional to the density: that is $\eecarreqx = 0$ and $\eecarreqy= 0$, and is linked with the energy through the relation $2\eecarrerho+4+3\eerho=0$. 
\end{prpn}
\begin{prpn}{7}
At equilibrium, the square of the energy is proportional to the density: that is $\eecarreqx = 0$ and $\eecarreqy= 0$, and both relaxation times related to the heat flux satisfy $\sigmafluxx =\sigmafluxy =1/(12\sigmatenseurxx)$.
\end{prpn}
\begin{prpn}{8}
Both relaxation times related to the heat flux satisfy $\sigmafluxx =\sigmafluxy = 1/(12\sigmatenseurxx)$, and both viscosities are linked by $\sigmae =\sigmatenseurxx$. 
\end{prpn}
Properties \ppns{1}{5}$+$\ppn{6} involve the following structure of the matrices $E$ and $S$:
\begin{eqnarray*}
E =\begin{pmatrix}
\eerho\lambda^2      &  0           &  0 \\
\dfrac{-4-3\eerho}{2}\lambda^4 &  0           &  0 \\
0                    &  -\lambda^2  &  0 \\
0                    &  0           &  -\lambda^2 \\
0                    &  0           & 0 \\
0                    &  0           &  0
\end{pmatrix}, \\
S = \operatorname{Diag}\trans{(\se,\secarre,\sfluxx,\sfluxy,\stenseurxx,\stenseurxx)}.
\end{eqnarray*}
Properties \ppns{1}{5}$+$\ppn{7} involve the following structure of the matrices $E$ and $S$:
\begin{eqnarray*}
E = \begin{pmatrix}
\eerho\lambda^2      &  0           &  0 \\
\eecarrerho\lambda^4 &  0           &  0 \\
0                    &  -\lambda^2  &  0 \\
0                    &  0           &  -\lambda^2 \\
0                    &  0           & 0 \\
0                    &  0           &  0
\end{pmatrix}, \\
S = \operatorname{Diag}\trans{\left(\se,\secarre,3\dfrac{2-\stenseurxx}{3-\stenseurxx},3\dfrac{2-\stenseurxx}{3-\stenseurxx},\stenseurxx,\stenseurxx\right)}.
\end{eqnarray*}
Properties \ppns{1}{5}$+$\ppn{8} involve the following structure of the matrices $E$ and $S$:
\begin{gather*}
E = \begin{pmatrix}
\eerho\lambda^2      &  0           &  0 \\
\eecarrerho\lambda^4 &  \eecarreqx\lambda^3           &   \eecarreqy\lambda^3\\
0                    &  -\lambda^2  &  0 \\
0                    &  0           &  -\lambda^2 \\
0                    &  0           & 0 \\
0                    &  0           &  0
\end{pmatrix},\\
S = \operatorname{Diag}\trans{\left(\stenseurxx,\secarre,3\dfrac{2-\stenseurxx}{3-\stenseurxx},3\dfrac{2-\stenseurxx}{3-\stenseurxx},\stenseurxx,\stenseurxx\right)}.
\end{gather*}
\end{lem}

\begin{proof}
We first consider eq. \eqq{eq11112cos2sin2} and we establish property \ppn{6}. Then, multiple choices appear considering eqs. \eqq{eq32223}, \eqq{eq22221} and \eqq{eq21112}. In order to give more details about these choices we have to define seven relations:

\begin{center}
\begin{tabular}{p{0.8cm}>{$}p{4.5cm}<{$}p{0.8cm}>{$}p{4.5cm}<{$}}
\ppn{a}& \eecarreqy = 0, & \ppn{d}&-1+6\sigmatenseurxx(\sigmafluxx+\sigmafluxy)=0, \\
\ppn{b}& \sigmae = \sigmatenseurxx, & \ppn{e}& 4+2\eecarrerho+3\eerho=0,\\
\ppn{c}& \eecarreqx=0, & \ppn{f}& -1+4\sigmatenseurxx(2\sigmafluxx+\sigmafluxy) = 0,\\
& & \ppn{g}& \sigmafluxx =\sigmafluxy.
\end{tabular}
\end{center}

The investigation of these equations involves some of these seven relations. More precisely, we have:\\
- eq. \eqq{eq32223cos1sin0} yields to a dichotomy between \ppn{a} and \ppn{c},\\
- eq. \eqq{eq32223cos0sin1} yields to a dichotomy between \ppn{b} and \ppn{c},\\
- eq. \eqq{eq32223cos5sin0} yields to a dichotomy between \ppn{a} and \ppn{d},\\
- eq. \eqq{eq32223cos4sin1} yields to a dichotomy between \ppn{c} and \ppn{d},\\
- eq. \eqq{eq22221cos1sin1} yields to a dichotomy between \ppn{e} and \ppn{f},\\
- eq. \eqq{eq21112cos0sin2} yields to a dichotomy between \ppn{e} and \ppn{g}.

That gives \textit{a priori} $2^8=256$ possibilities. However it is straightforward that only three cases remain: \\
- properties \ppn{a}, \ppn{c} and \ppn{e} are true,\\
- properties \ppn{a}, \ppn{c}, \ppn{f} and \ppn{g} are true,\\
- properties \ppn{b}, \ppn{d} and \ppn{f} are true. 

Finally, each of these three cases solves all of the reminded equations and we get isotropy at the third order. 
\end{proof}
\begin{rmk}\ \\
- The properties \ppns{5}{6} are already used in the literature (the equality $2\eecarrerho+4+3\eerho=0$ is a generalization of the classical choice $\eerho=-2$ and $\eecarrerho=1$), so it is really interesting to obtain a justification of this choice by the study of the isotropy. \\
- The equality $\sfluxx = 3(2-\stenseurxx)/(3-\stenseurxx)$ is also proposed in \cite{LL00}. \\
- To our knowledge, the sets of parameters given by properties \ppns{1}{5}-\ppn{7} and \ppns{1}{5}-\ppn{8} are new though the second one is less interesting for acoustic applications because of the link between both viscosities. 
\end{rmk}

%%%%%%%%%%%%%%%%%%%%%%%%%%%%%%%%%%
\subsubsection{Fourth order}
%%%%%%%%%%%%%%%%%%%%%%%%%%%%%%%%%%

For the \ddqn scheme, the isotropy at fourth order can be characterized by the following Lemma:

\begin{lem}
\label{lem:d2q9o4}
The \ddqn scheme is isotropic at fourth order if and only if the whole of the properties \ppns{1}{6}, \ppns{9}{11} and one of the two properties  \ppns{12}{13} are satisfied.
\begin{prpn}{9}
 The relaxations times related to odd moments satisfy $\sigma_\fluxx = \sigma_\fluxy =1/(6\sigma_\tenseurxx)$.
\end{prpn}
\begin{prpn}{10}
At equilibrium, the square of the kinetic energy is proportional to the density $\ecarre = \eecarrerho\rho\lambda^4$. 
\end{prpn}
\begin{prpn}{11}
At equilibrium, the kinetic energy and its square are linked by $ 2\eecarrerho+4+3\eerho=0$.
\end{prpn}
\begin{prpn}{12}
 The relaxations times related to even moments satisfy  $\sigmae=\sigmaecarre=\sigma_\tenseurxx$.
\end{prpn}
\begin{prpn}{13}
 The sound velocity is imposed through the equality $2+3\eerho=0$. 
\end{prpn}
These conditions involve the following structure for the matrices $E$ and $S$:
\begin{eqnarray*}
E = \begin{pmatrix}
\eerho\lambda^2\  &  0  &  0\\
\dfrac{-4-3\eerho}{2}\lambda^4 &  0                &  0\\
0         &  -\lambda^2        &  0\\
0         & 0        &  -\lambda^2 \\
0 &  0  &  0\\
0 &  0  &  0
\end{pmatrix},\\
S = \operatorname{Diag}\trans{\left(\se,\secarre,6\dfrac{2-\stenseurxx}{6-\stenseurxx},6\dfrac{2-\stenseurxx}{6-\stenseurxx},\se,\se\right)},
\end{eqnarray*}
where either $\secarre=\se$ or $\eerho=-2/3$. 
\end{lem}

\begin{proof}
Since there are three sets of coefficients giving third order isotropy, the proof is divided into three cases. 

The cases with \ppns{1}{5} and either \ppn{7} or \ppn{8} are forbidden because they both imply the condition $\sigmatenseurxx=0$, using eq. \eqq{eq111111cos2sin2}. 

So we have to assume \ppns{1}{6} in order to impose isotropy at third order. These assumptions imply $\eecarreqx=\eecarreqy=0$ using eqs.~\eqq{eq111222cos5sin0} and \eqq{eq111222cos4sin1} and $\sigmafluxy=\sigmafluxx$ by eq. \eqq{eq212222cos5sin1}. Then eq. \eqq{eq311113cos2sin2} involves property \ppn{9} and eq.~\eqq{eq112221cos1sin1} involves property \ppn{11}. Finally the dichotomy between \ppn{12} and \ppn{13} comes from eq.~\eqq{eq211112cos2sin2}. 
\end{proof}

\begin{rmk}\ \\
- This Lemma is the last step to prove that the heat flux and the momentum are collinear: the property of isotropy at fourth order imposes the equality between both relaxations times $\sfluxx$ and $\sfluxy$. 
 Moreover, it links these two relaxation times with 
 the other one $\stenseurxx$. \\
- The property \ppn{11} is involved in a particular case in \cite{LL00} ($\eerho=-2$, $\eecarrerho=1$) but without the constraint on the relaxation times given in \ppn{9}.\\
- The case \ppns{1}{5}-\ppns{9}{11}-\ppn{13} insures isotropy at fourth order. However, the combination of relation $\eerho=6c_0^2-4$ and property \ppn{13} yields to $c_0^2=5/9$, and this value for the sound velocity involves that the bulk viscosity disappears by using eq.~\eqref{eq:viscosities} \cite{D01}.\\
- The isotropy at fourth order is quite restrictive: both relaxation times linked to the viscosities ($\se$ and $\stenseurxx$) must be equal.\\
- Note that these results generalize those obtained in \cite{DL09} for quartic parameters. More precisely, the values $\sigmafluxx=\sqrt{3}/3$ and $\sigmatenseurxx=\sqrt{3}/{6}$ proposed in \cite{DL09} are compatible with property \ppn{9}.
\end{rmk}

%%%%%%%%%%%%%%%%%%%%%%%%%%%%%%
\subsubsection{Fifth order}
%%%%%%%%%%%%%%%%%%%%%%%%%%%%%%

For the \ddqn scheme, we have the following Lemma:
\begin{lem}
\label{lem:d2q9o5}
The \ddqn scheme is never isotropic at fifth order. 
\end{lem}
\begin{proof}
 In both cases \ppns{1}{5}-\ppns{9}{12} and  \ppns{1}{5}-\ppns{9}{11}-\ppn{13}, Proposition~\ref{prop:isotropy} gives equations that can not be solve for every rotation of the frame: for example \eqq{eq1111112} is of type $c\cos\theta\sin\theta$, where $c$ is a given real constant independent of the parameters (namely, $c$ do not vanish). 
\end{proof}

%%%%%%%%%%%%%%%%%%%%%%%%%%%%%%%%%%%%%%%%%%%%%%%%%%
\subsection{\ddqt scheme}
%%%%%%%%%%%%%%%%%%%%%%%%%%%%%%%%%%%%%%%%%%%%%%%%%%
\label{sec:d2q13}

In this Subsection, we give the main result on the \ddqt scheme in Proposition~\ref{prop:d2q13}. The proof of this Proposition is then subdivided into two Lemmas in order to detail the methodology due to the definition of the isotropy. 

\begin{prop}
\label{prop:d2q13}
 Let $L_4(r)$ be the lack of isotropy for the \ddqt scheme at fourth order for the rotation $r$, then we get the following propositions.
\begin{itemize}
 \item The scheme is isotropic at first order, that is $L_4(r)=\grandO({\Delta t}^2),\qq r\in\rotation[2]$, iff $\eeqx$, $\eeqy$, $\etenseurxxrho$, $\etenseurxxqx$, $\etenseurxxqy$, $\etenseurxyrho$, $\etenseurxyqx$, and $\etenseurxyqy$ are zero.
 \item The scheme is isotropic at second order, that is $L_4(r)=\grandO({\Delta t}^3),\qq r\in\rotation[2]$,
iff the four following properties are satisfied:
\begin{itemize}
\item it is isotropic at first order, 
\item $\efluxxrho$, $\efluxyrho$, $\execarrerho$, and $\eyecarrerho$ are zero,  
\item $\efluxxqx$ and $\efluxyqy$ are equal to $(\sigmatenseurxx-4\sigmatenseurxy)/(3\sigmatenseurxy+\sigmatenseurxx)$, 
\item $\execarreqx=\eyecarreqy$, $\efluxxqy=-\efluxyqx$, and $\execarreqy=-\eyecarreqx$.
\end{itemize}
\end{itemize}
\end{prop}

The proof of Proposition \ref{prop:d2q13} is detailed below thanks to two Lemmas: one for each order. The methodology is the same as the one for the \ddqn scheme. 

%%%%%%%%%%%%%%%%%%%%%%%%%%%%%%%%%%%%%%%%%%%%%%%%%%
\subsubsection{First order}
%%%%%%%%%%%%%%%%%%%%%%%%%%%%%%%%%%%%%%%%%%%%%%%%%%

For the \ddqt scheme, the isotropy at first order is characterized by the following Lemma:

\begin{lem}
\label{lem:d2q13o1}
The \ddqt scheme is isotropic at the first order if and only if the properties \ppts{1}{2} are satisfied.
\begin{prpt}{1}
At equilibrium, the energy is proportional to the density (the proportionality factor is relative to the sound velocity).
\end{prpt}
\begin{prpt}{2}
At equilibrium, both moments $\mtenseurxx$ and $\mtenseurxy$ are zero.
\end{prpt}
Both properties \ppts{1}{2} implies the following structure of the matrix $E$ (there is no constraint on the matrix $S$ at first order):
\begin{equation*}
E = \begin{pmatrix}
\eerho\lambda^2\  &  0  &  0\\
0 &  0  &  0\\
0 &  0  &  0\\
\efluxxrho\lambda^3        &  \efluxxqx\lambda^2        &  \efluxxqy\lambda^2\\
\efluxyrho\lambda^3        &  \efluxyqx\lambda^2        &  \efluxyqy\lambda^2\\
\execarrerho\lambda^5        &  \execarreqx\lambda^4        &  \execarreqy\lambda^4\\
\eyecarrerho\lambda^5        &  \eyecarreqx\lambda^4        &  \eyecarreqy\lambda^4\\
\eecarrerho\lambda^4 &   \eecarreqx\lambda^3               &  \eecarreqy\lambda^3\\
\eecuberho\lambda^6 &   \eecubeqx\lambda^5                & \eecubeqy\lambda^5 \\
\exxerho\lambda^4   &   \exxeqx\lambda^3		 &   \exxeqy\lambda^3
\end{pmatrix}.
\end{equation*}
\end{lem}
\begin{rmk}
 The first order isotropy conditions are the same for the \ddqn and the \ddqt schemes. 
\end{rmk}

%%%%%%%%%%%%%%%%%%%%%%%%%%%%%%%%%%%%%%%%%%%%%%%%%%
\subsubsection{Second order}
%%%%%%%%%%%%%%%%%%%%%%%%%%%%%%%%%%%%%%%%%%%%%%%%%%

For the \ddqt scheme, the isotropy at second order is described in the following Lemma:
\begin{lem}
\label{lem:d2q13o2}
The \ddqt scheme is isotropic at second order if and only if the properties \ppts{1}{5} are satisfied.
\begin{prpt}{3}
At equilibrium the heat flux is a rotation-dilatation of the momentum, more precisely $\mfluxx=\efluxxqx\lambda^2q_x+\efluxxqy\lambda^2q_y$ and $\mfluxy=-\efluxxqy\lambda^2q_x+\efluxxqx\lambda^2q_y$.
\end{prpt}
\begin{prpt}{4}
At equilibrium the moment of order five $(\mxecarre,\myecarre)$ is a rotation-dilatation of the momentum, more precisely
$\mxecarre=\execarreqx\lambda^4q_x+\execarreqy\lambda^4q_y$ and $\myecarre=-\execarreqy\lambda^4q_x+\execarreqx\lambda^4q_y$.
\end{prpt}
\begin{prpt}{5}
 The equilibrium states are related to the relaxation times by the relations $\execarreqx = a$ and $\execarreqy = b$, with
\begin{equation}\label{eq:defab}
  a :=  -\dfrac{1}{12}\dfrac{7(7\sigmatenseurxx{+}2\sigmatenseurxy)\efluxxqx+5(17\sigmatenseurxx{-}4\sigmatenseurxy)}{\sigmatenseurxx+\sigmatenseurxy}, \,
 b := \dfrac{7}{12}\efluxyqx\dfrac{7\sigmatenseurxx+2\sigmatenseurxy}{\sigmatenseurxx+\sigmatenseurxy}.
\end{equation}
\end{prpt}
Properties \ppts{1}{5} implies the following structure of the matrix $E$ (there is no constraint the matrix $S$ at second order):
\begin{equation*}
E = \begin{pmatrix}
\eerho\lambda^2\  &  0  &  0\\
0 &  0  &  0\\
0 &  0  &  0\\
0        &  \efluxxqx\lambda^2        &  \efluxxqy\lambda^2\\
0        &  -\efluxxqy\lambda^2        &  \efluxxqx\lambda^2\\
0        &  a\lambda^4        &  b\lambda^4\\
0        &  -b\lambda^4        &  a\lambda^4\\
\eecarrerho\lambda^4 &   \eecarreqx\lambda^3               &  \eecarreqy\lambda^3\\
\eecuberho\lambda^6 &   \eecubeqx\lambda^5                & \eecubeqy\lambda^5 \\
\exxerho\lambda^4   &   \exxeqx\lambda^3		 &   \exxeqy\lambda^3
\end{pmatrix},
\end{equation*}
where $a$ and $b$ are given in \eqref{eq:defab}. 
\end{lem}

\begin{rmk}
The properties of isotropy at second order for the \ddqn and \ddqt schemes are of the same type,
 though they are less constraining for the \ddqt scheme (the coefficient $\efluxxqx$ does not have to be zero for instance).
\end{rmk}

%%%%%%%%%%%%%%%%%%%%%%%%%%%%%%%%%%%%%%%%%%%%%%%%%%
\subsubsection{Third order}
%%%%%%%%%%%%%%%%%%%%%%%%%%%%%%%%%%%%%%%%%%%%%%%%%%

Because of the very high number of cases offered to guarantee the isotropy at third order
(to our knowledge, not less than 17 different cases have to be investigated), 
only some sufficient conditions are given in this paper.
First of all, we already know that the isotropy at third order implies properties \ppts{1}{9}.

\begin{prpt}{7}
At equilibrium, the heat flux is proportional to the momentum: that is $\mfluxx=\efluxxqx\lambda^2qx$ and $\mfluxy=\efluxxqx\lambda^2qy$.
\end{prpt}

\begin{prpt}{8}
At equilibrium the moment of order five $(\mxecarre,\myecarre)$ is proportional to the momentum: that is $\mxecarre=\execarreqx\lambda^4qx$ and $\mfluxy=\execarreqx\lambda^4qy$.
\end{prpt}

\begin{prpt}{9}
 The proportional coefficients $\efluxxqx$ and $\execarreqx$ are linked by the relation $\eecarreqx=-(21/8)\efluxxqx-65/24$. 
\end{prpt}

Assuming properties \ppts{1}{9}, we give some example of sets of coefficients that make the \ddqt scheme isotropic in the Annex. 

%%%%%%%%%%%%%%%%%%%%%%%%%%%%%%%%%%%%%%%%%%%%%%%%%%
\section{Numerical results for the \ddqn scheme}
%%%%%%%%%%%%%%%%%%%%%%%%%%%%%%%%%%%%%%%%%%%%%%%%%%
\label{sec:numeric}

In this Section, we present some preliminary numerical results in order to appreciate the lack of isotropy order after order. Since it is not easy to observe the lack of isotropy in the sense of Definition~\ref{def:isotropy}, we have to investigate an other way to make explicit the isotropy error. 

Then, it seems natural to represent the evolution of a radial function after a few time steps. Let $r\mapsto\rho(r,\theta)$ be the density solution of the linear lattice Boltzmann \ddqn scheme described in Section~\ref{sec:d2q9} and initialized with the Gaussian for the first moment $\rho(x,y) = \exp(-10x^2 - 10y^2)$ and with zero for the others. Indeed the most the scheme is isotropic, the most the behavior of $r\mapsto\rho(r,\theta)$ is the same for all of the angle $\theta$.  We consider a \ddqn scheme with $100\times100$ space meshing (in our case $\Delta x = 0.02$ and $\lambda = 1$), our purpose is to make explicit the isotropy error at $t=12\Delta t$. Thus, we plot the density $r\mapsto\rho(r,\theta)$ where $\theta$ is fixed:
\begin{itemize}
 \item $\rho_0:=r\mapsto\rho(r,{\theta=0})$ (drawn with $\vee$),
 \item $\rho_{\frac{\pi}{2}}:=r\mapsto\rho(r,{\theta=\frac{\pi}{2}})$ (drawn with $\wedge$),
 \item $\rho_{\frac{\pi}{4}}:=r\mapsto\rho(r,{\theta=\frac{\pi}{4}})$ (drawn with $\circ$),
 \item $\rho_{\arctan(1/2)}:=r\mapsto\rho(r,{\theta=\arctan(1/2)})$ (drawn with $<$).
\end{itemize}
For each density, we only have the values on the nodes of the mesh: namely for each angle $\theta$, the vectors of the abscissa of $(r,\rho_\theta)$ are not equal. Since we have to compare these densities, we interpolate the value with a spline method of order 5 so that the interpolation error does not interfere.

We first consider the following parameters for which the scheme is isotropic at first order but not at second order:
$$E  = 
 \begin{pmatrix}
{-2}\  &  0  &  0\\
 {6}&  0  &  0 \\
0         &  {-2}        &  0\\
 0         & 0        &  {-2} \\

 0 &  0  &  0\\
 0 &  0                &  0
 \end{pmatrix}\quad \text{and}\quad 
S  = \operatorname{diag}\begin{pmatrix}
                        \se\\{0.5\se} \\ \sfluxx\\ \sfluxx\\ {\stenseurxx}\\ {\stenseurxx}
                       \end{pmatrix},
$$
where $\zeta= 1.84\,10^{-5}$, $\mu=2/3\zeta$, $\se = 3\zeta/(\lambda\Delta x)=1.9977944349438221$, $\sfluxx =1.3$, $\stenseurxx=3\mu/(\lambda\Delta x)=1.9985290825952098$.

The different curves $\rho_\theta$ are plotted in Figure~\ref{fig:12}. The isotropy of the scheme is weak in the sense that, for instance, the plot of $\rho_0-\rho_{\frac{\pi}{4}}$ in Figure~\ref{fig:err12} gives an error of size 6.5E-4. We remark that the curves of $\rho_0$ and $\rho_{\frac{\pi}{2}}$ are the same: the scheme is invariant by a rotation of $\pi/2$. %The same remark also holds for the densities  $\rho_{\arctan(1/2)}$ and $\rho_{\arctan(1/2)}$. 

\begin{figure}[htbp]
\begin{minipage}{0.495\textwidth}
\begin{center}
  \includegraphics[width=\textwidth]{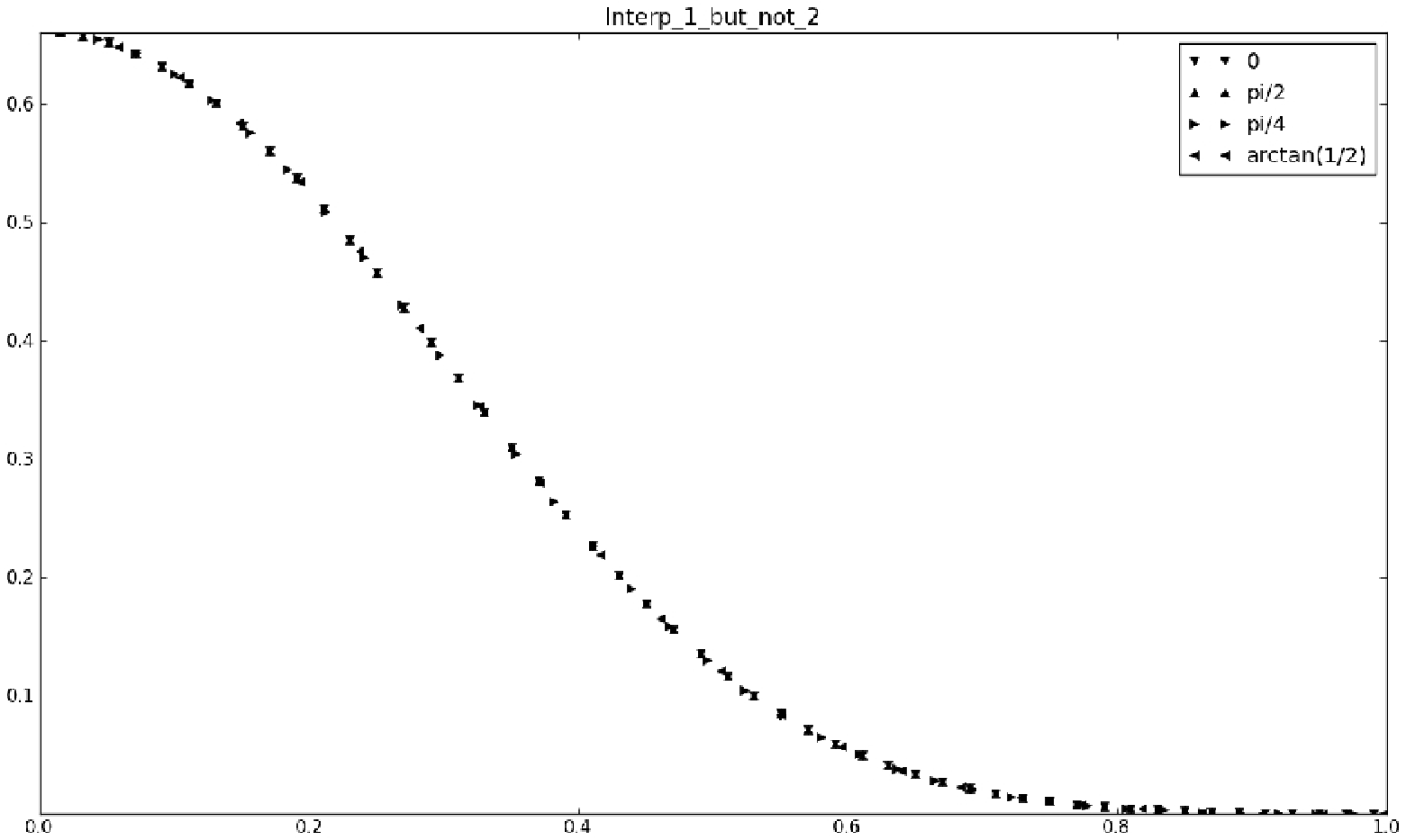}
  \caption{Isotropy at first order}
  \label{fig:12}
\end{center}
\end{minipage}
\begin{minipage}{0.495\textwidth}
\begin{center}
  \includegraphics[width=\textwidth]{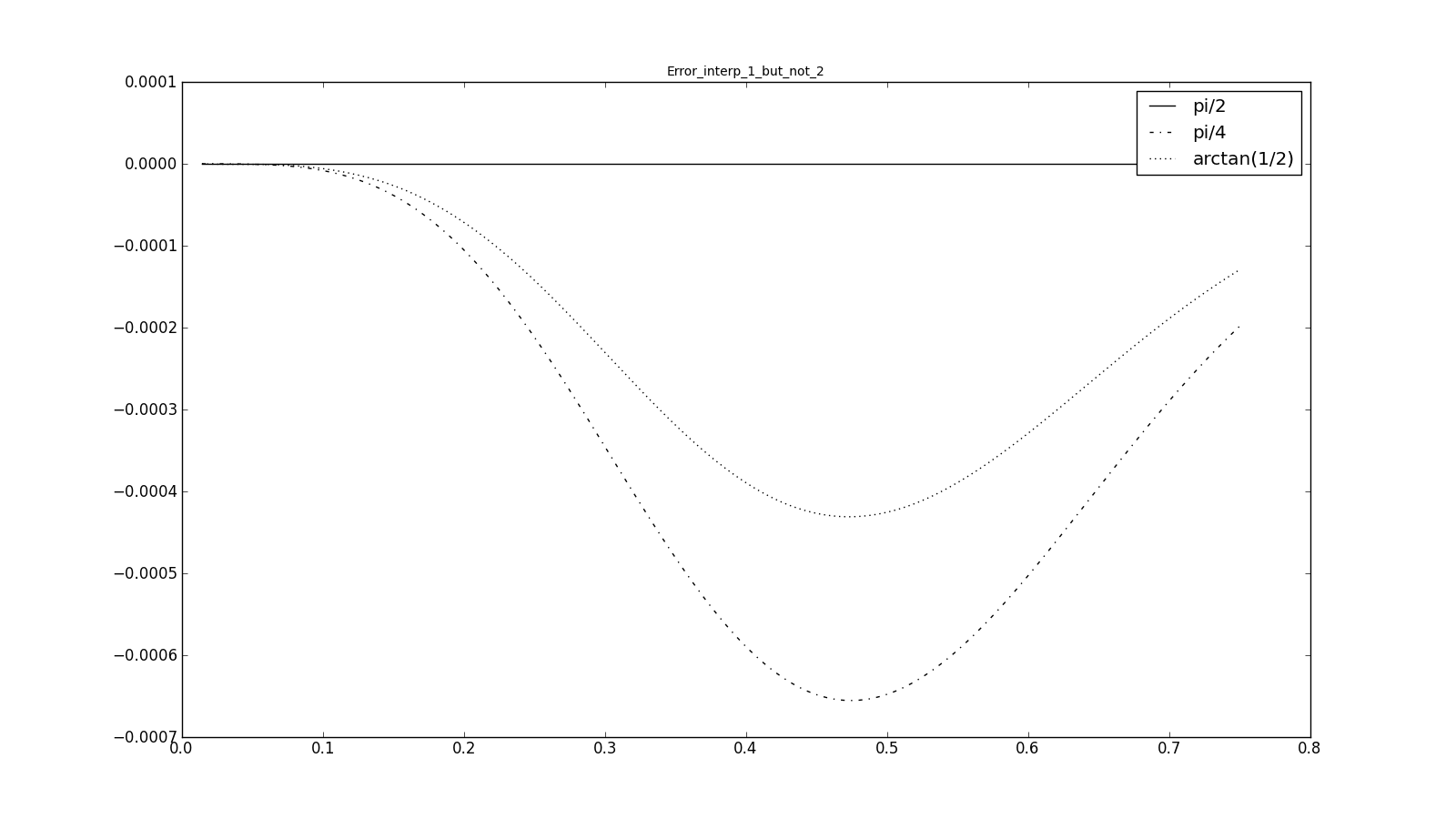}
  \caption{Isotropy error at first order}
\label{fig:err12}
\end{center}
\end{minipage}
\end{figure}

We then consider the following parameters, with no change on the relaxations times, for which the scheme is isotropic at second order but not at third order:
$$E  = 
 \begin{pmatrix}
{-2}\  &  0  &  0\\
 {6}&  0  &  0\\
 0         &  {-1}        &  0\\
 0         & 0        &  {-1} \\
 0 &  0  &  0\\
 0 &  0                &  0
\end{pmatrix}\quad \text{and}\quad 
S  = \operatorname{diag}\begin{pmatrix}
                        \se\\ 0.5\se\\ \sfluxx\\\sfluxx\\ {\stenseurxx}\\ {\stenseurxx}
                       \end{pmatrix}.
$$

In Figure~\ref{fig:23}, we observe that this choice of parameters improves the isotropy in the sense that the different curves are more similar but the difference between $\rho_0$ and $\rho_{\frac{\pi}{4}}$ is still of order 4.5E-4 on Figure~\ref{fig:err23}.

\begin{figure}[htbp]
\begin{minipage}{0.495\textwidth}
\begin{center}
  \includegraphics[width=\textwidth]{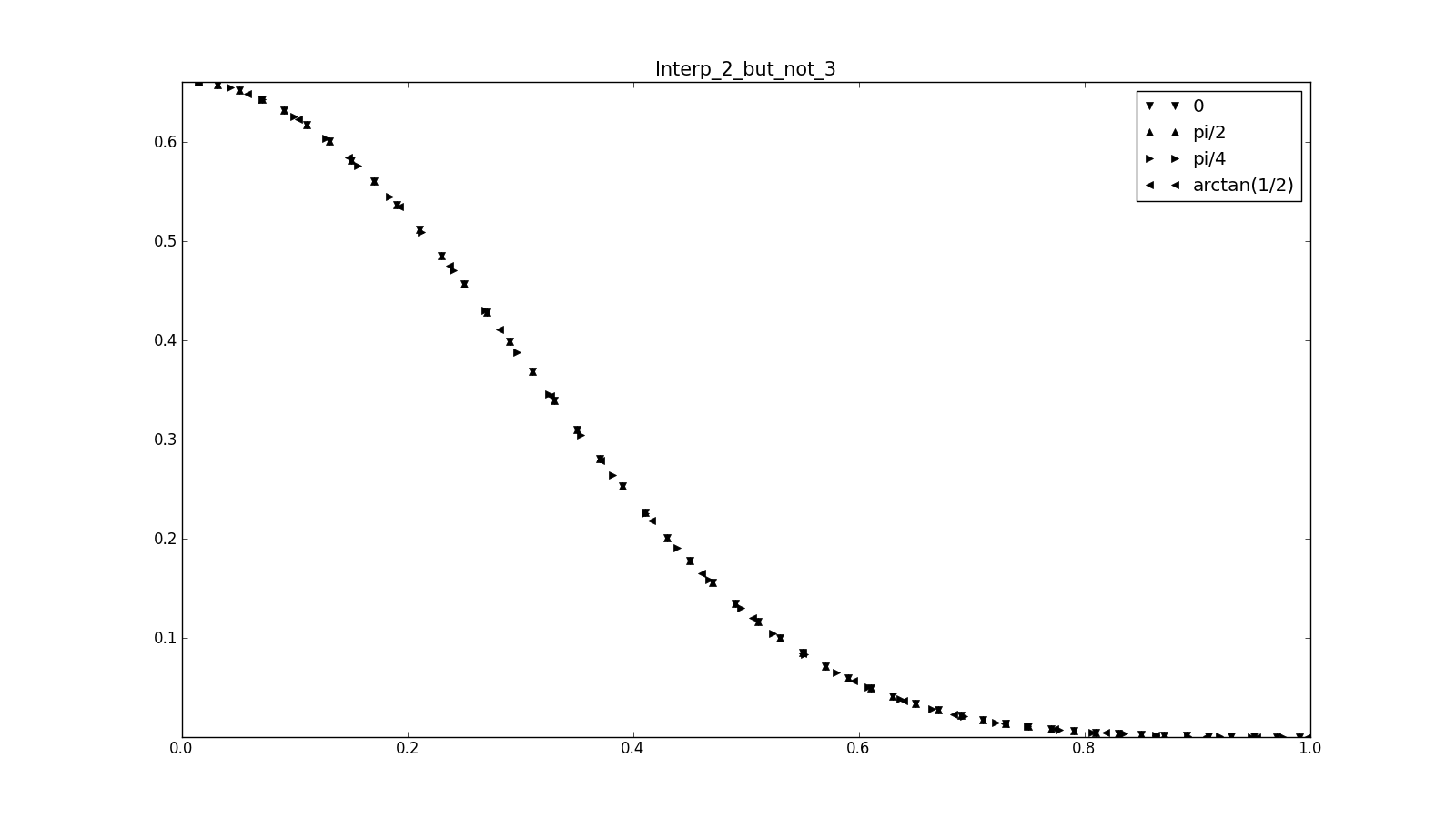}
  \caption{Isotropy at second order}
  \label{fig:23}
\end{center}
\end{minipage}
\begin{minipage}{0.495\textwidth}
\begin{center}
  \includegraphics[width=\textwidth]{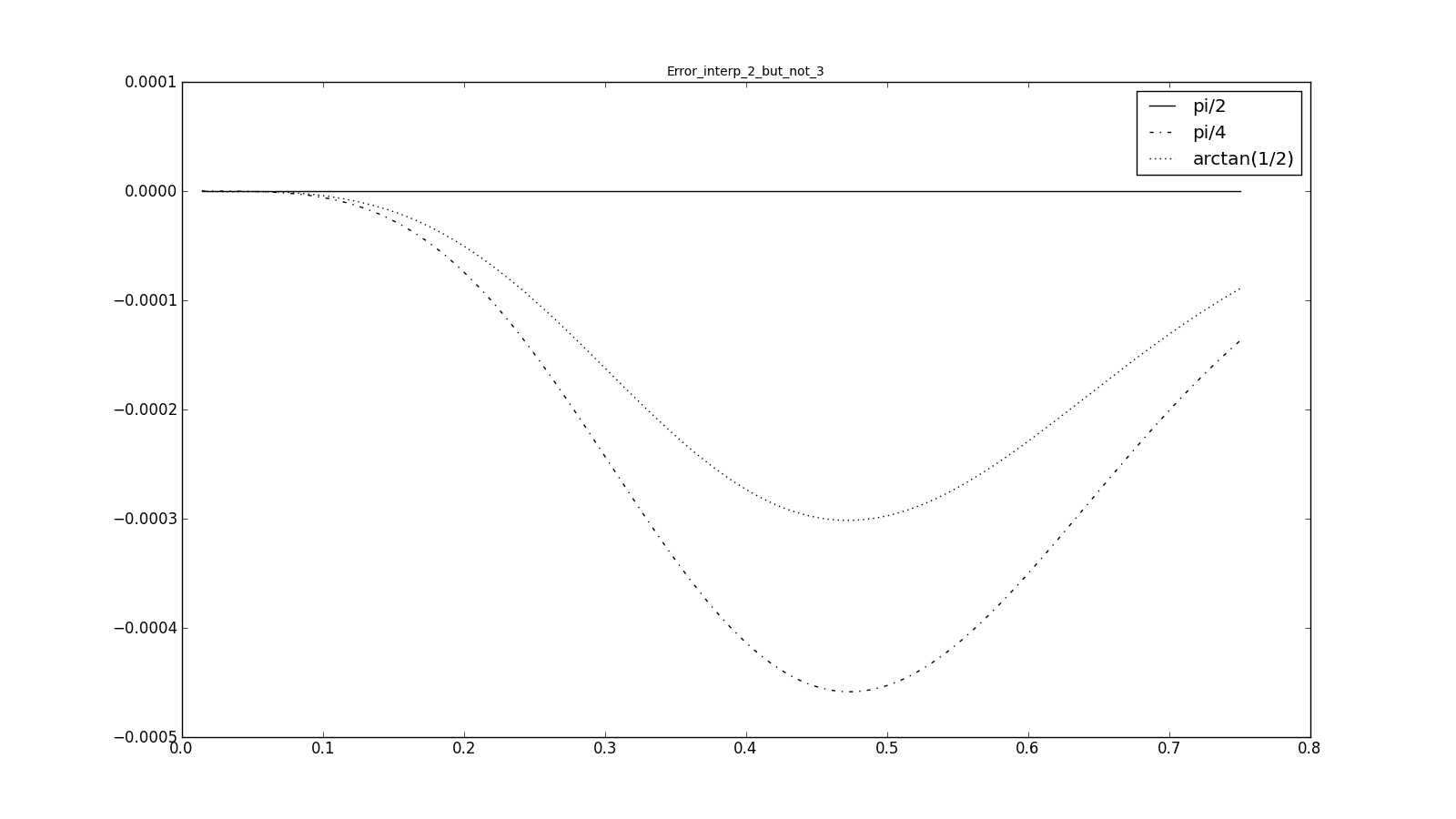}
  \caption{Isotropy error at second order}
\label{fig:err23}
\end{center}
\end{minipage}
\end{figure}

We next consider the following parameters, with no change on the relaxations times, for which the scheme is isotropic at third order but not at fourth order:
$$E  = 
 \begin{pmatrix}
{-2}\  &  0  &  0\\
 {1}&  0  &  0\\
 0         &  {-1}        &  0\\
 0         & 0        &  {-1} \\
 0 &  0  &  0\\
 0 &  0                &  0
\end{pmatrix}\quad \text{and}\quad 
S  = \operatorname{diag}\begin{pmatrix}
                        \se\\{0.5\se}\\ \sfluxx\\\sfluxx\\ {\stenseurxx}\\ {\stenseurxx}
                       \end{pmatrix}.
$$

The results are given in  Figure~\ref{fig:34} is better than the one given for isotropy only at the second order: the plot of $\rho_0-\rho_{\frac{\pi}{4}}$ in Figure~\ref{fig:err34} gives an error of size 4.2E-6. 

\begin{figure}[htbp]
\begin{minipage}{0.495\textwidth}
\begin{center}
  \includegraphics[width=\textwidth]{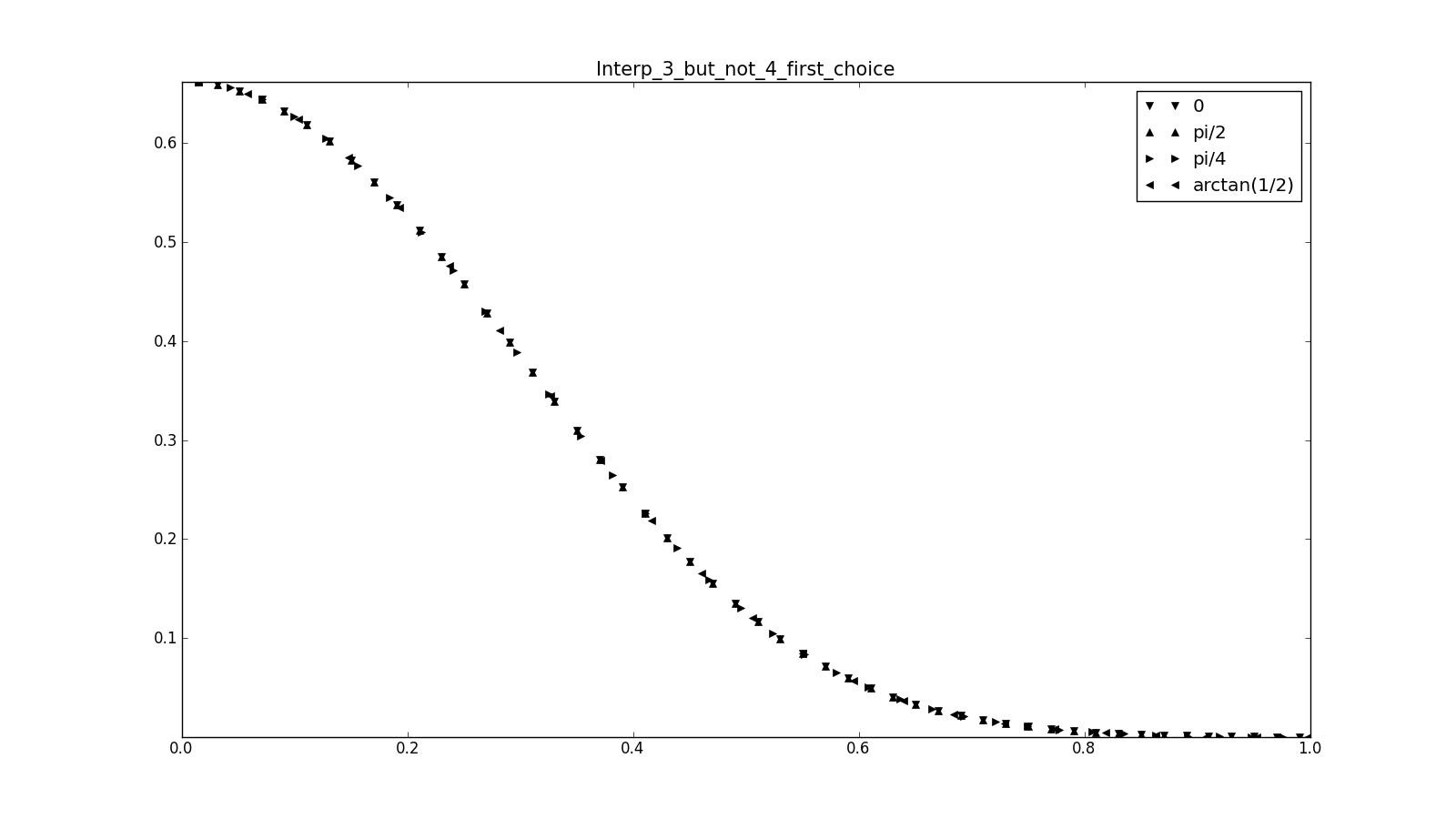}
  \caption{Isotropy at third order}
  \label{fig:34}
\end{center}
\end{minipage}
\begin{minipage}{0.495\textwidth}
\begin{center}
  \includegraphics[width=\textwidth]{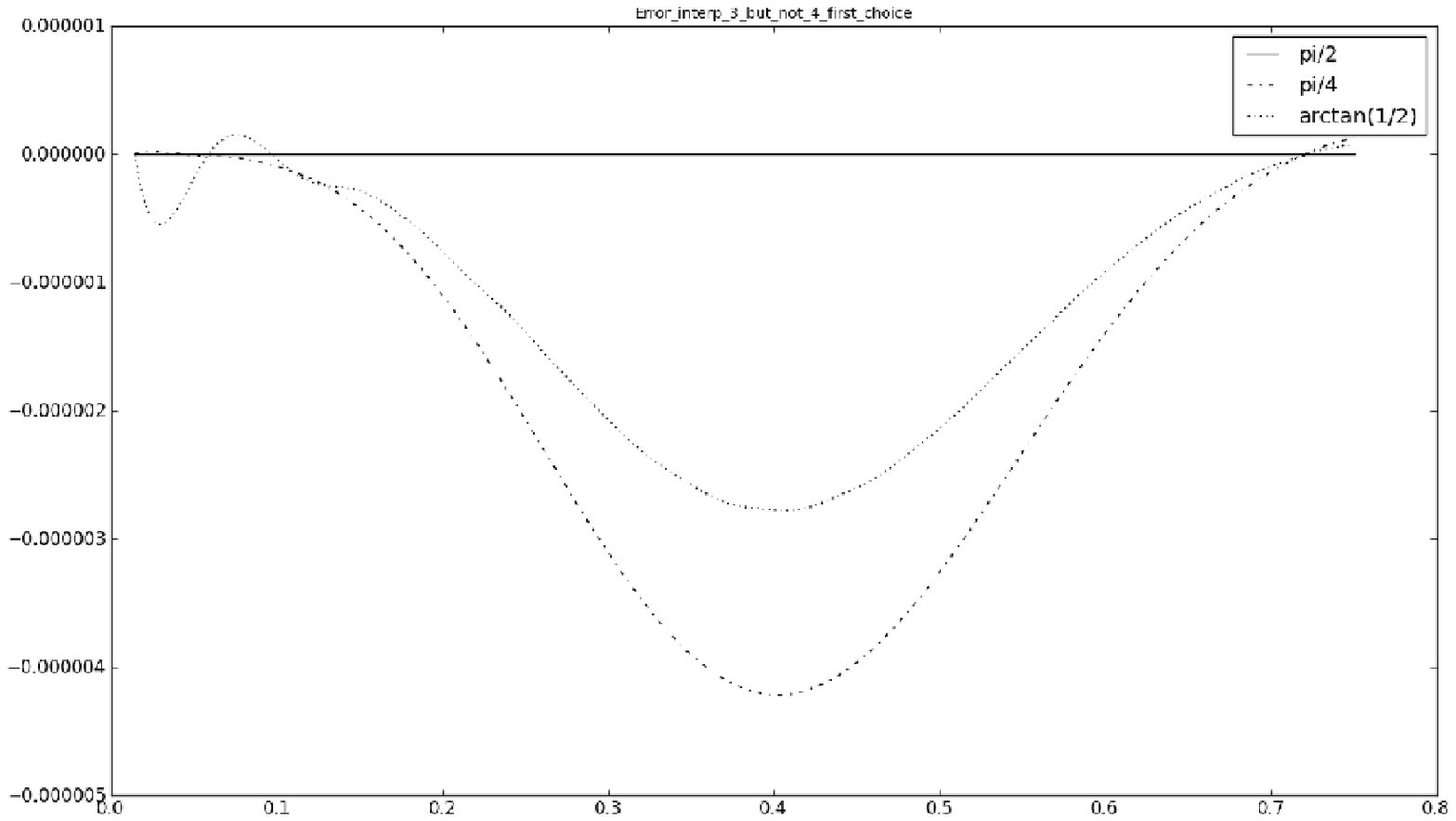}
  \caption{Isotropy error at third order}
\label{fig:err34}
\end{center}
\end{minipage}
\end{figure}

And finally, we consider the following parameters for which the scheme is isotropic at fourth order:
$$E  = 
 \begin{pmatrix}
{-2}\  &  0  &  0\\
 {1}&  0  &  0\\
 0         &  {-1}        &  0\\
 0         & 0        &  {-1} \\
 0 &  0  &  0\\
 0 &  0                &  0
\end{pmatrix}\quad \text{and}\quad 
S  = \operatorname{diag}\begin{pmatrix}
                        \se\\ \se\\ \sfluxx\\\sfluxx\\ {\se}\\ {\se}
                       \end{pmatrix}.
$$
where $\zeta= 1.84\,10^{-5}=\mu$, $\se = 3\zeta/(\lambda\Delta x)=1.9977944349438221$, $\sfluxx = 6.0(2.0-\stenseurxx)/(6.0-\stenseurxx)=0.0022055650561781941$.

The result is given in Figure~\ref{fig:4} and  the plot of $\rho_0-\rho_{\frac{\pi}{4}}$ in Figure~\ref{fig:err4} gives an error of size 5.5E-7. 

\begin{figure}[htbp]
\begin{minipage}{0.495\textwidth}
\begin{center}
  \includegraphics[width=\textwidth]{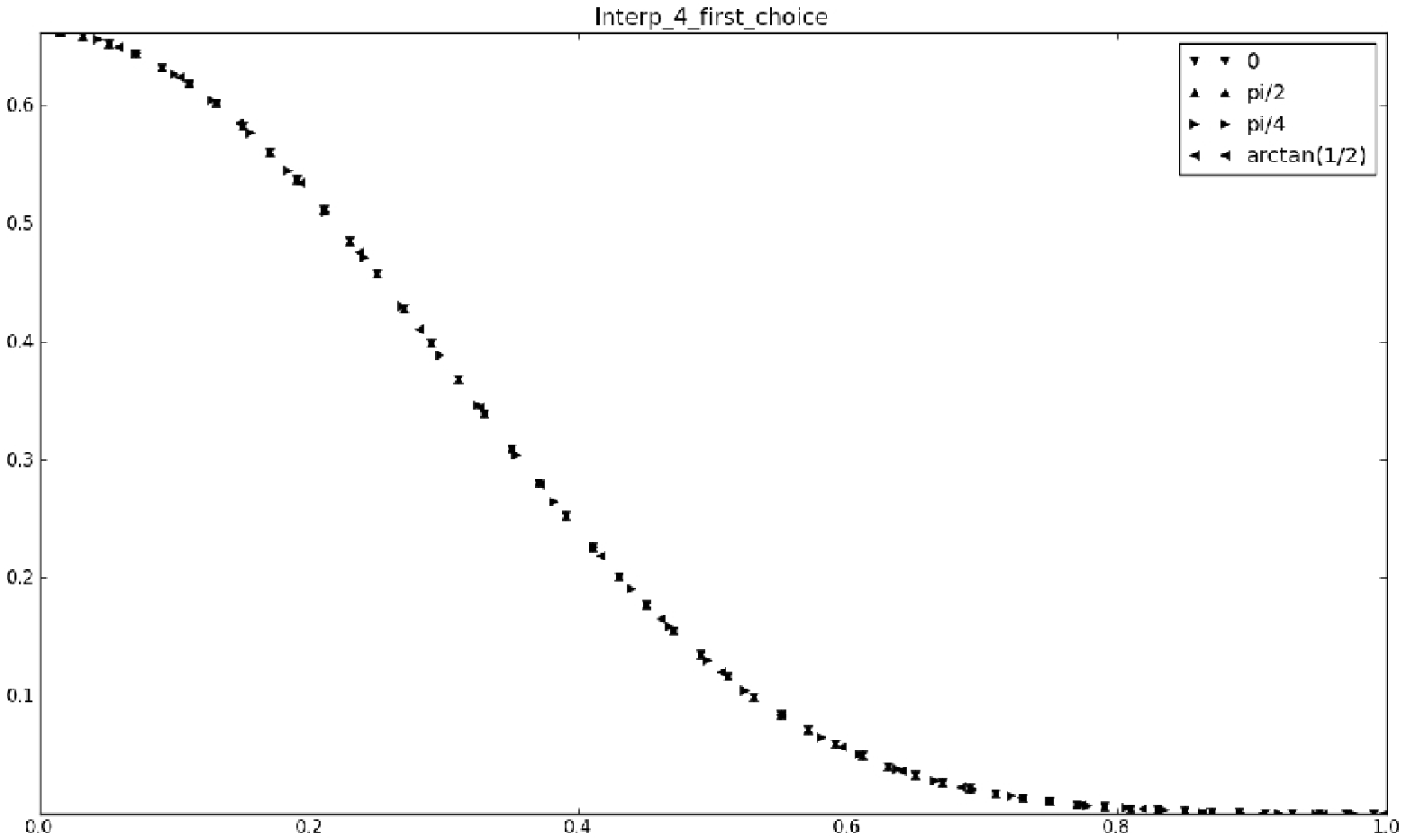}
  \caption{Isotropy at fourth order}
  \label{fig:4}
\end{center}
\end{minipage}
\begin{minipage}{0.495\textwidth}
\begin{center}
  \includegraphics[width=\textwidth]{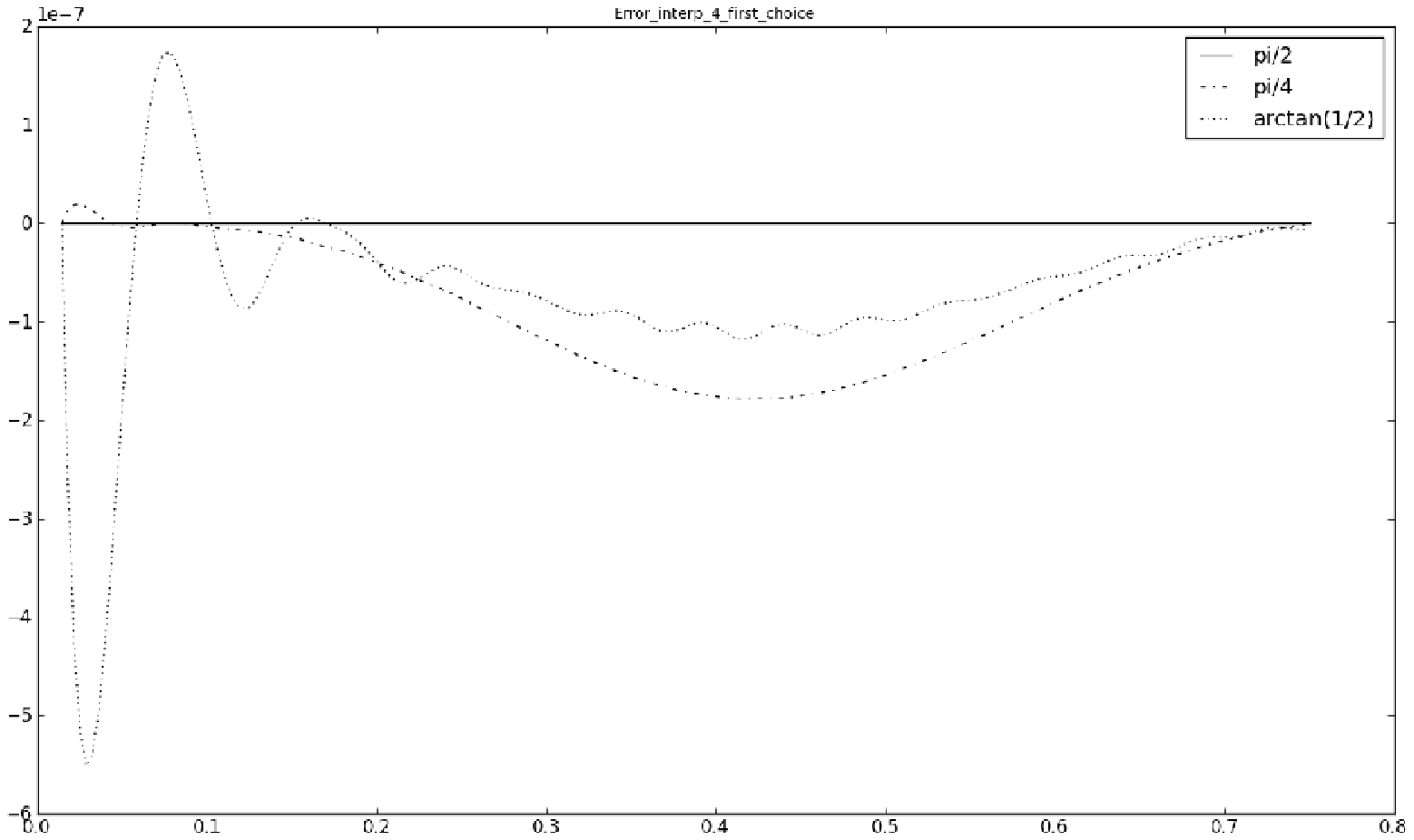}
  \caption{Isotropy at fourth order}
\label{fig:err4}
\end{center}
\end{minipage}
\end{figure}

%%%%%%%%%%%%%%%%%%%%%%%%%%%%%%%%%%%%%%%%%%%%%%%%%%
\section*{Conclusion}
%%%%%%%%%%%%%%%%%%%%%%%%%%%%%%%%%%%%%%%%%%%%%%%%%%
We used a general methodology that guarantees isotropy of a lattice Boltzmann scheme for a given order. 
This approach is based on the expansion of the equivalent PDEs at an arbitrary order and the invariance by the group operation $\Phi_n$. We have detailed all the possible cases for the basic scheme \ddqn applied to a linearized fluid mechanics. Results (up to the second order) have also been proposed for the \ddqt scheme. Using \guill{inappropriate} choices of parameters in the \ddqn scheme, elementary test cases hightlight the lack of isotropy at various orders. This work can be extended without conceptional difficulties for 3D lattice Boltzmann schemes and it will be done in a forthcoming contribution. 

%%%%%%%%%%%%%%%%%%%%%%%%%%%%%%%%%%%%%%%%%%%%%%%%%%
\section*{Annex}
%%%%%%%%%%%%%%%%%%%%%%%%%%%%%%%%%%%%%%%%%%%%%%%%%%
In this annex, we propose some details concerning the isotropy at third order for the \ddqt scheme. This work is in progress:  we express the different cases that have to be study. 

In fact, knowing properties \ppts{1}{9}, eqs.~\eqq{eq11113} and \eqq{eq22233} give multiple choices between the thirteen following relations: 
\begin{center}
\begin{tabular}{p{0.7cm}>{$}p{3.8cm}<{$}p{0.8cm}>{$}p{6.5cm}<{$}}
\ppt{a}& \efluxxqx = -3, & \ppt{f2}& \mbox{Second value for }\eecubeqx, \\
\ppt{b}& \sigmatenseurxy =\sigmatenseurxx, & \ppt{g}&3(\sigmatenseurxx+\sigmatenseurxy)(\sigmaxecarre+\sigmayecarre)=1,\\
\ppt{c1}& \mbox{First value for }\eecubeqy ,& \ppt{f3}&\mbox{Third value for }\eecubeqx,\\
\ppt{d}& \sigmayecarre = 1/12\sigmatenseurxx, &\ppt{g2}&6\sigmatenseurxy(\sigmaxecarre+\sigmayecarre)+3\sigmatenseurxx(\sigmayecarre+3\sigmaxecarre),\\
\ppt{e}& \sigmaxecarre = 1/12\sigmatenseurxx,&\ppt{c3}&\mbox{Third value for }\eecubeqy ,\\
\ppt{f1}& \mbox{First value for }\eecubeqx,&\ppt{c4}&\mbox{Fourth value for }\eecubeqy.\\
\ppt{c2}&  \mbox{Second value for }\eecubeqy ,& &
\end{tabular}
\end{center}

The investigation of these equations involves some of these thirteen relations. More precisely, we have:\\
- eq. \eqq{eq11113cos1sin1} yields to a dichotomy between \ppt{a} and \ppt{b},\\
- eq. \eqq{eq32223cos0sin0} yields to a dichotomy between \ppt{c1} and \ppt{d},\\
- eq. \eqq{eq32223cos0sin1} yields to a dichotomy between \ppt{e} and \ppt{f1},\\
- eq. \eqq{eq32223cos1sin0} yields to a dichotomy between \ppt{c2} and \ppt{e},\\
- eq. \eqq{eq32223cos4sin1} yields to a dichotomy between \ppt{f2} and \ppt{g1},\\
- eq. \eqq{eq32223cos2sin1} yields to a dichotomy between \ppt{f3} and \ppt{g2},\\
- eq. \eqq{eq32223cos3sin0} yields to a dichotomy between \ppt{c3} and \ppt{g2},\\
- eq. \eqq{eq32223cos5sin0} yields to a dichotomy between \ppt{g1} and \ppt{c4}.

That gives \textit{a priori} $2^8=256$ possibilities. 
However preliminary calculations yield that only seventeen cases remain: in order to have isotropy of third order, it is necessary to satisfy at least one of these sets of properties. \\
1)  \ppt{a}, \ppt{c3}, \ppt{c4}, \ppt{d}, \ppt{e}, \ppt{f2} and \ppt{f3} \\
2)  \ppt{b}, \ppt{d}, \ppt{e}, \ppt{g1} and \ppt{g2} \\
3)  \ppt{a}, \ppt{c2}, \ppt{c3}, \ppt{c4}, \ppt{d}, \ppt{f1}, \ppt{f2} and \ppt{f3} \\
4)  \ppt{a}, \ppt{c2}, \ppt{c4}, \ppt{d}, \ppt{f1}, \ppt{f2} and \ppt{g2} \\
5)  \ppt{a}, \ppt{c2}, \ppt{c3}, \ppt{d}, \ppt{f1}, \ppt{f3} and \ppt{g1} \\
6)  \ppt{a}, \ppt{c1}, \ppt{c3}, \ppt{c4},  \ppt{e}, \ppt{f2} and \ppt{f3} \\
7)  \ppt{a}, \ppt{c1}, \ppt{c4},  \ppt{e}, \ppt{f2} and  \ppt{g2} \\
8)  \ppt{a}, \ppt{c1}, \ppt{c3},  \ppt{e}, \ppt{f3}  and \ppt{g1} \\
9)  \ppt{a}, \ppt{c1}, \ppt{c2}, \ppt{c3}, \ppt{c4},  \ppt{f1}, \ppt{f2} and \ppt{f3} \\
10)  \ppt{a}, \ppt{c1}, \ppt{c2}, \ppt{c4},  \ppt{f1}, \ppt{f2} and \ppt{g2} \\
11)  \ppt{a}, \ppt{c1}, \ppt{c2}, \ppt{c3},  \ppt{f1}, \ppt{f3} and \ppt{g1} \\
12)  \ppt{a}, \ppt{c1}, \ppt{c2},  \ppt{f1}, \ppt{g1} and \ppt{g2} \\
13)  \ppt{b}, \ppt{c1}, \ppt{c2}, \ppt{c3}, \ppt{c4},  \ppt{f1}, \ppt{f2} and \ppt{f3} \\
14)  \ppt{b}, \ppt{c1}, \ppt{c2}, \ppt{c4},  \ppt{f1}, \ppt{f2} and \ppt{g2} \\
15)  \ppt{b}, \ppt{c1}, \ppt{c2}, \ppt{c3},  \ppt{f1}, \ppt{f3} and \ppt{g1} \\
16)  \ppt{b}, \ppt{c2}, \ppt{c3}, \ppt{c4},  \ppt{d}, \ppt{f1}, \ppt{f2} and \ppt{f3} \\
17)  \ppt{b}, \ppt{c1}, \ppt{c3}, \ppt{c4}, \ppt{e},  \ppt{f2} and \ppt{f3}

The study of  these cases is in progress and we cannot then give a full characterization of it. However 
we know some examples of matrices $E$ and $S$ that involve isotropy at third order. 
We propose here two cases given by properties \ppt{a}, \ppt{c3}, \ppt{c4}, \ppt{d}, \ppt{e}, \ppt{f2} and \ppt{f3}.

\begin{eqnarray*}
E = \begin{pmatrix}
\eerho\lambda^2\  &  0  &  0\\
0 &  0  &  0\\
0 &  0  &  0\\
0        &  -3\lambda^2        &  0\\
0        &  0       &  -3\lambda^2\\
0        &  \dfrac{31}{6}\lambda^4        &  0\\
0        &  0        &   \dfrac{31}{6}\lambda^4\\
\eecarrerho\lambda^4 &   \eecarreqx\lambda^3               &  \eecarreqy\lambda^3\\
\left(\dfrac{274}{39}-\dfrac{67}{462}\eecarrerho-\dfrac{137}{3003}\eerho\right)\lambda^6 &   -\dfrac{67}{462}\eecarreqx\lambda^5                & -\dfrac{67}{462}\eecarreqy\lambda^5 \\
\exxerho\lambda^4   &   \exxeqx\lambda^3		 &   \exxeqy\lambda^3
\end{pmatrix},\\
S = \mbox{diag}\trans{\left(\stenseurxx,\stenseurxx,\stenseurxy,\sfluxx,\sfluxx,\sfluxx,\sfluxx,\secarre,\secube,\sxxe\right)},\\ \mbox{ where }\sfluxx = 3\dfrac{2-\stenseurxx}{3-\stenseurxx},
\end{eqnarray*}

 or for instance

\begin{eqnarray*}
&E = \begin{pmatrix}
\eerho\lambda^2\  &  0  &  0\\
0 &  0  &  0\\
0 &  0  &  0\\
0        &  -3\lambda^2        &  0\\
0        &  0       &  -3\lambda^2\\
0        &  \dfrac{31}{6}\lambda^4        &  0\\
0        &  0        &   \dfrac{31}{6}\lambda^4\\
\left(-\dfrac{3234}{13}-\dfrac{361}{26}\eerho\right)\lambda^4 &   0               &  0\\
\left(\dfrac{1681}{39}+\dfrac{307}{156}\eerho\right)\lambda^6 &   0                & 0 \\
0   &   0	 &   0
\end{pmatrix},&\\ 
&S = \mbox{diag}\trans{\left(\se,\stenseurxx,\stenseurxy,2\frac{7\se+5\stenseurxx-6\stenseurxx\se}{7\se+5\stenseurxx-4\stenseurxx\se},\sfluxy,\sxecarre,\sxecarre,\secarre,\secube,\sxxe\right)},&\\
&\mbox{ where }\sxecarre = 3\dfrac{2-\stenseurxx}{3-\stenseurxx}.&
\end{eqnarray*}

\begin{rmk}
 Since the \ddqt scheme takes into account the velocities of the \ddqn and 4 additional ones, it can be read as a generalization of the \ddqn scheme. In order to illustrate this remark, we propose here a set of coefficients that gives isotropy at third order for both \ddqn and \ddqt schemes:
\begin{eqnarray*}
&E = \begin{pmatrix}
\eerho\lambda^2\  &  0  &  0\\
0 &  0  &  0\\
0 &  0  &  0\\
0        &  -\lambda^2        &  0\\
0        &  0       &  -\lambda^2\\
0        &  -\dfrac{1}{12}\lambda^4        &  0\\
0        &  0        &   -\dfrac{1}{12}\lambda^4\\
\left(-\dfrac{3234}{13}-\dfrac{361}{26}\eerho\right)\lambda^4 &   0               &  0\\
\left(\dfrac{1681}{39}+\dfrac{307}{156}\eerho\right)\lambda^6 &   0                & 0 \\
0   &   0	 &   0
\end{pmatrix},&\\ 
&S = \mbox{diag}\trans{\left(\stenseurxx,\stenseurxx,\stenseurxx,\fluxx,\sfluxx,\sfluxx,\sfluxx,\secarre,\secube,\sxxe\right)},&\\ 
&\mbox{ where }\sfluxx = 3\dfrac{2-\stenseurxx}{3-\stenseurxx}.&
\end{eqnarray*}
\end{rmk}

%%%%%%%%%%%%%%%%%%%%%%%%%%%%%%%%%%%%%%%%%%%%%%%%%%
\section*{Acknowledgments}
%%%%%%%%%%%%%%%%%%%%%%%%%%%%%%%%%%%%%%%%%%%%%%%%%%
This work has been financially supported by the French Ministry of Industry (DGCIS) and the Region Ile-de-France in the framework of the LaBS Project \cite{labs}.

The authors are very grateful to Li-Shi Luo for his really interesting remark during the ICMMES that showed that the presented results at the Conference were very preliminary. 
%%%%%%%%%%%%%%%%%%%%%%%%%%%%%
% Biblio
%%%%%%%%%%%%%%%%%%%%%%%%%%%%%
\section*{References}
\bibliographystyle{plain}
\nocite{}

\bibliography{ADG_bibliographie}
\end{document}